%% file: FillableSurgery.tex
\documentclass[11pt]{amsart}

\usepackage{epsfig,amsmath,amsfonts,latexsym,mathrsfs,amssymb,flafter}
\usepackage{color}
\usepackage{overpic}
 \usepackage{palatino}
 \usepackage{hyperref}
 \usepackage{longtable}
 \usepackage[nocompress]{cite}

\newtheorem{theorem}{Theorem}[section]
\newtheorem{proposition}[theorem]{Proposition}
\newtheorem{lemma}[theorem]{Lemma}
\newtheorem{corollary}[theorem]{Corollary}

\theoremstyle{definition}
\newtheorem{definition}[theorem]{Definition}
\newtheorem{remark}[theorem]{Remark}
\newtheorem{example}[theorem]{Example}

\theoremstyle{remark}

\newcommand{\ds}{\displaystyle}

\newcommand{\zee}{\mathbb{Z}}
\newcommand{\arr}{\mathbb{R}}
\newcommand{\cue}{\mathbb{Q}}
\newcommand{\cee}{\mathbb{C}}

\newcommand{\K}{{\mathcal K}}

\newcommand{\T}{{\mathcal T}}
\newcommand{\C}{{\mathcal C}}

\newcommand{\ts}{\textstyle}

\newcommand{\cptwobar}{{\overline{\cee P}^2}}
\newcommand{\tDelta}{{\widetilde{\Delta}}}
\DeclareMathOperator{\tb}{tb}
\DeclareMathOperator{\rot}{rot}

\DeclareMathOperator{\Mod}{Mod}

\DeclareMathOperator{\id}{id}
\DeclareMathOperator{\Br}{Br}

\DeclareFontFamily{U}{mathx}{\hyphenchar\font45}
\DeclareFontShape{U}{mathx}{m}{n}{
      <5> <6> <7> <8> <9> <10>
      <10.95> <12> <14.4> <17.28> <20.74> <24.88>
      mathx10
      }{}
\DeclareSymbolFont{mathx}{U}{mathx}{m}{n}
\DeclareFontSubstitution{U}{mathx}{m}{n}
\DeclareMathAccent{\widecheck}{0}{mathx}{"71}

\theoremstyle{plain}

\begin{document}

\title{Fillable contact structures from positive surgery}

\author{Thomas E. Mark}

\author{B\"{u}lent Tosun}

\address{Department of Mathematics \\ University of Virginia \\ Charlottesville \\ VA}

\email{tmark@virginia.edu}

\address{Department of Mathematics\\ University of Alabama\\Tuscaloosa\\AL}

\email{btosun@ua.edu}

\begin{abstract} We consider the question of when the operation of contact surgery with positive surgery coefficient, along a knot $K$ in a contact 3-manifold $Y$, gives rise to a weakly fillable contact structure. We show that this happens if and only if $Y$ itself is weakly fillable, and $K$ is isotopic to the boundary of a properly embedded symplectic disk inside a filling of $Y$. Moreover, if $Y'$ is a contact manifold arising from positive contact surgery along $K$, then any filling of $Y'$ is symplectomorphic to the complement of a suitable neighborhood of such a disk in a filling of $Y$.

Using this result we deduce several necessary conditions for a knot in the standard 3-sphere to admit a fillable positive surgery, such as quasipositivity and equality between the slice genus and the 4-dimensional clasp number, and we give a characterization of such knots in terms of a quasipositive braid expression.  We show that knots arising as the closure of a positive braid always admit a fillable positive surgery, as do knots that have lens space surgeries, and suitable satellites of such knots. In fact the majority of quasipositive knots with up to 10 crossings admit a fillable positive surgery. On the other hand, in general (strong) quasipositivity, positivity, or Lagrangian fillability need not imply a knot admits a fillable positive contact surgery.
\end{abstract}

\maketitle

\section{Introduction}
Every closed, oriented 3-dimensional manifold $Y$ admits contact structures, and any contact structure on $Y$ can be obtained from the standard contact structure on $S^3$ by contact surgery along a Legendrian link. Many geometric properties of contact structures are preserved by contact surgery with negative (contact) surgery coefficient, such as tightness and the various notions of fillability, so it is natural to consider the effect of positive contact surgery on these properties. Here we study the case of surgery along a single knot $K\subset Y$ and ask when the result of contact surgery, with some positive contact surgery coefficient, yields a fillable contact structure. Theorem \ref{mainthm} below provides, in a certain sense, a complete answer to this question.

 Recall that a compact symplectic 4-manifold $(Z,\omega)$ is a weak symplectic filling of $(Y,\xi)$ if $\partial Z = Y$ as oriented manifolds, and $\omega$ is positive on the contact planes $\xi$ (here, and throughout, we are supposing that $\xi$ has a given orientation). We say $(Z,\omega)$ is a strong filling of $(Y,\xi)$ if there is a Liouville vector field defined in a neighborhood of $Y$ and directed out of $Z$, whose induced contact structure is $\xi$.

 \begin{theorem}\label{mainthm} Let $\K\subset Y$ be an oriented Legendrian knot in a closed contact 3-manifold $(Y,\xi)$. For an integer $n$, let $\xi_n^-(\K)$ denote the contact structure obtained by contact $n$-surgery along $\K$ using ``all negative stabilizations.'' Then there exists $n>0$ such that $\xi_n^-(\K)$ is weakly symplectically filled by some symplectic manifold $(Z', \omega')$, if and only if both of the following hold:
\begin{itemize}
\item $(Y, \xi)$ is weakly symplectically fillable.
\item Some weak filling $(Z,\omega)$ of $(Y,\xi)$ contains a properly embedded symplectic disk $\Delta$ such that $\partial \Delta$ is a positive transverse pushoff of $\K$.
\end{itemize}
Moreover, in this situation $Z$ and $\Delta$ can be chosen so that $(Z',\omega')$ is symplectomorphic to the complement of a suitable neighborhood of $\Delta$ in $Z$.
\end{theorem}

We will use the phrase ``$\K$ admits a fillable positive surgery'' as shorthand for the condition that some $n>0$ exists as in the theorem. In fact, if $K\subset Y$ is a smooth knot, we will say $K$ admits a fillable positive surgery if there exists some Legendrian representative $\K$ of $K$ with that property. Note that any smooth knot admits many inequivalent Legendrian representatives, but the condition for fillable surgery depends only on the transverse pushoff of such representatives. In particular, since oriented Legendrians $\K$ and $\K'$ differing by negative stabilization(s) have equivalent transverse pushoffs, the question of admitting fillable positive surgery has the same answer for $\K$ and $\K'$. As we will review below (Remark \ref{techremark}), the fillable contact structures resulting from such surgeries are contactomorphic so long as the underlying smooth surgeries are equivalent.

The forward implication of the ``if and only if'' statement of the theorem was proved by Conway, Etnyre, and the second author in \cite[Theorem 1.13]{CET} (that theorem is stated only for knots in the standard 3-sphere, but the same argument proves the more general case). We give a slightly different proof in Section \ref{proofsec} (see Theorem \ref{capthm}). The converse direction is a consequence of the following result.

\begin{theorem}\label{scoopthmintro} Let $(Z,\omega)$ be a weak symplectic filling of $(Y, \xi)$, and $\Delta\subset Z$ a properly embedded symplectic disk with positively transverse boundary $K\subset Y$. Then $\Delta$ has an arbitrarily small tubular neighborhood $U$ such that if $Z' = Z - U$ and $\omega' = \omega|_{Z'}$, then $(Z', \omega')$ is a weak symplectic filling of a contact structure $\xi'$ on $Y' = \partial Z'$. Moreover, the contact structure $\xi'$ is obtained from $\xi$ by an inadmissible transverse surgery along $K$, or (equivalently) from a positive contact surgery along a Legendrian approximation of $K$.
\end{theorem}

We remark that if $\K\subset Y$ has the property that $\xi^-_n(\K)$ is weakly fillable for some $n>0$, then $\xi^-_m(\K)$ is fillable for all $m\geq n$. There are various ways to see this; a proof from the current point of view is easily supplied by symplectically blowing up points of the symplectic disk $\Delta\subset Z$ provided by Theorem \ref{mainthm} (see the discussion in Section \ref{surgcoeffsec}). Thus, for given $\K$, the set of integers $n>0$ such that $\xi^-_n(\K)$ is weakly fillable is either empty or an interval of the form $[n_0, \infty)$.
 
When specialized to knots in the standard contact structure $\xi_{std}$ on $S^3$, Theorem \ref{mainthm} provides very effective means to determine whether a knot $K$ admits a fillable positive surgery. (We remark that in the case $(Y, \xi) = (S^3, \xi_{std})$, weak and strong fillability of $\xi_n^-(\K)$ are equivalent. See Corollary \ref{strongfillcor}.) Indeed, the symplectic fillings of $(S^3, \xi_{std})$ are all symplectomorphic to blowups of the 4-ball, so the question reduces to asking whether $K$ is the boundary of a symplectic disk inside such a blowup. 

On the side of constraints we have the following corollary, where we recall that the {\it 4-dimensional clasp number} $c_*(K)$ is the minimum number of double points in any generically immersed disk in $B^4$ with boundary $K$. The {\it slice genus} $g_*(K)$ is the minimum genus of any smoothly embedded surface in $B^4$  with boundary $K$; by smoothing double points we have the obvious inequality $g_*(K)\leq c_*(K)$ for any $K\subset S^3$. In general, the difference between $c_*(K)$ and $g_*(K)$ can be arbitrarily large \cite{daemiscaduto:clasp, JZ:clasp}. The definition of quasipositivity is reviewed in Section \ref{equivcondsec}.

\begin{corollary}\label{c*g*cor} If $K\subset (S^3, \xi_{std})$ admits a fillable positive surgery, then $K$ must satisfy:
\begin{itemize}
\item $K$ is quasipositive.
\item The invariants $g_*(K)$ and $c_*(K)$ are equal.
\end{itemize}
Moreover, $K$ bounds a properly immersed disk in $B^4$ with only transverse double point singularities, having $c_*(K) = g_*(K)$  positive double points, and no negative double points. 
\end{corollary}

The proof of Corollary \ref{c*g*cor} is spelled out in Section \ref{S3proofsec}. It immediately implies the following, which in particular answers Question 2 of \cite{CET}.

\begin{corollary}\label{counterexcor} There exist Legendrian knots $\K$ in $S^3$  such that $\xi_n^-(\K)$ is tight for all sufficiently large $n$, but $\xi^-_n(\K)$ is not fillable for any $n>0$. In fact, there exist such $\K$ among positive knots, in particular those that bound Lagrangian surfaces embedded in $B^4$.
\end{corollary}

Indeed, one need only consider the knot $K = 7_4$ in the standard tables. This knot has $g_*(K) = 1$, but it was shown by Owens-Strle \cite{OwSt2016} that $c_*(K) = 2$. Hence $K$ admits no fillable positive surgery by Corollary \ref{c*g*cor}. On the other hand, $K$ is a positive knot in the sense that it admits a diagram with only positive crossings, hence is Lagrangian fillable by \cite{HaySab}, and it is not hard to see that for $\K$ a Legendrian representative with maximal Thurston-Bennequin number, the contact surgeries $\xi^-_n(\K)$ are tight for all $n\geq 1$ (for example, it follows from \cite{Golla} or \cite{MT:surgery} that the Heegaard Floer contact invariants of these contact structures are all nonzero). Note that positive knots are in particular strongly quasipositive \cite{Rudolph99}, so Corollary \ref{counterexcor} also shows that strong quasipositivity does not imply a fillable positive surgery.

In the other direction, in Section \ref{equivcondsec} we provide general and comparatively computable conditions under which a knot in $S^3$ admits a fillable positive surgery. These conditions are given in Theorems \ref{fillsurgS3thm} and \ref{qpslicethm}. For the sake of illustration in this introduction, we have the following result on existence of fillable positive surgeries for particular classes of knots.

\begin{corollary}\label{existencecor}
A knot $K\subset S^3$ admits a fillable positive surgery if it satisfies any of the following conditions.
\begin{enumerate}
\item[(a)] $K$ is isotopic to the closure of a positive braid.
\item[(b)] Some smooth surgery on $K$ (with positive surgery coefficient) yields a lens space.
\item[(c)] $K$ is obtained as a twisted satellite $P_m(C)$, where:
\begin{itemize}
\item $C\subset S^3$ is a knot that admits a fillable positive surgery, 
\item $m$ is the smooth surgery coefficient corresponding to some such fillable surgery on $C$, and 
\item the pattern $P$ is a braided fillable pattern in the sense of Definition \ref{braidpatterndef}.
\end{itemize}
\item[(d)] In particular a cabled knot $K = C_{p,q}$ admits a fillable positive surgery, so long as $C$ itself has that property and the cabling parameters satisfy $q/p > m$, where $m$ is as in (c).

\end{enumerate}
\end{corollary}

It is interesting to note that while all {\it known} knots that admit a lens space surgery are also isotopic to the closure of positive braids (up to mirroring; this is observed, for example, in \cite{GodaTera2000}), our proofs of cases (a) and (b) are different. In particular, part (b) along with Corollary \ref{c*g*cor} implies that any knot $K$ with a lens space surgery must  have $g_*(K) = c_*(K)$, which does not seem to be otherwise known in general. Part (c) implies the same property for many $L$-space knots, but {\it a priori} not all; see the discussion in Section \ref{constructionsec}.

The class of knots admitting fillable positive surgeries is much larger than might be suggested by Corollary \ref{existencecor}, however. In Section \ref{equivcondsec} we give a necessary and sufficient condition for fillable positive surgery in terms of a quasipositive braid expression of a knot; this criterion together with the constraints of Corollary \ref{c*g*cor} suffices to prove the following.

\begin{corollary}\label{tablecor} Of the 59 quasipositive knots of up to 10 crossings listed in the KnotInfo database \cite{knotinfo},  48 admit a fillable positive surgery while 11 do not.
\end{corollary}

Note that only 6 of the knots referenced in this corollary are lens space knots; these and two others are the only positive braids.

A variety of additional necessary conditions, constructions, and consequences related to fillable positive surgeries for knots in $S^3$ are obtained in Section \ref{S3proofsec}. Among these, stated in Corollary \ref{negdefcor}, is that if $K\subset S^3$ admits a fillable positive surgery with smooth surgery coefficient $r>0$, then every filling of the corresponding contact structure has negative definite intersection form. In particular, the surgery manifold $S^3_r(K)$ is the boundary of a smooth negative definite 4-manifold, which is a nontrivial constraint. In \cite{OwSt2012}, Owens and Strle consider the question of when a positive surgery on a knot $K\subset S^3$ has this property, and define in this context the invariant
\[
m(K) = \inf\{ r\in \cue_{>0}\,|\, \mbox{$S^3_r(K)$ bounds a negative definite 4-manifold}\}.
\]
They prove that $m(K)$ exists for all $K$ (i.e., that for any knot, some positive surgery bounds a negative definite 4-manifold). In a similar spirit, we can define
\[
\mu(K) = \min\{r\in \zee_{\geq 0}\,|\, \mbox{$K$ admits a fillable positive surgery with smooth coefficient $r$}\},
\]
where if no fillable positive surgery exists then we set $\mu(K) = \infty$. Then Corollary \ref{negdefcor} implies that for any $K\subset S^3$ we have $m(K)\leq \mu(K)$. Strictly, the case $\mu(K) =0$ requires a special argument: if $\mu(K) = 0$ then necessarily $K$ is slice, as follows from the discussion on fillable surgery coefficients at the beginning of Section \ref{surgcoeffsec}, and if $K$ is slice it is not hard to show that $m(K) = 0$.

A basic property of $\mu(K)$ is as follows:

\begin{proposition}\label{muboundprop} If $\mu(K)< \infty$ then we have
\[
2g_*(K) \leq \mu(K) \leq 4g_*(K).
\]
The first inequality is strict unless $g_*(K) = 0$.
\end{proposition}

This proposition is proved in Section \ref{surgcoeffsec}. Equality can occur for the second inequality in Proposition \ref{muboundprop}, as follows from the next result.

\begin{proposition}\label{torusknotprop} For the positive torus knot $T(p,q)$ we have $\mu(T(p,q)) = \lceil m(T(p,q))\rceil$. 
\end{proposition}

This is proved in Theorem \ref{torusknotthm}. The value of $m(T(p,q))$ was computed by Owens and Strle (and reviewed in Section \ref{surgcoeffsec}), and in particular we have $m(T(2,2n+1)) = 4g_*(T(2,2n+1) = 4n$. Hence the second inequality of Proposition \ref{muboundprop} is an equality in this case. For a ``typical'' torus knot the inequality is strict; for example we have
\begin{align*}
 \mu(T(p,p-1)) &= (p-1)^2 \\
  4g_*(T(p,p-1)) &= 2(p-1)(p-2).
\end{align*}

It follows from these examples that in general, the smallest fillable surgery coefficient $\mu(K)$ is not, for instance, determined by the genus or slice genus of $K$. 

\begin{corollary} If $K$ is a knot that admits a fillable positive surgery, then $m(K)\leq 4g_*(K)$.
\end{corollary}

Note that Owens-Strle show that for general $K$ one has the bound $m(K)\leq 4u_+(K)$, where $u_+$ is the minimum number of changes of positive crossings in an unknotting sequence for $K$. A similar argument shows $m(K)\leq 4c_+(K)$, where $c_+(K)$ is the minimal number of positive self-intersections in a normally immersed disk in $B^4$ with boundary $K$. In general, these bounds can be arbitrarily larger than $4g_*(K)$ \cite{daemiscaduto:clasp,JZ:clasp}, while as we have seen $m(K) = 4g_*(K)$ can occur. In fact, since there are examples of knots with fillable positive surgery for which $u_+(K) = u(K) > g_*(K)$ (such as the knot $10_{126}$), the bound in the Corollary above is sharper than the unknotting bound. On the other hand, of course, if $K$ admits a fillable positive surgery then $g_*(K) = c_*(K)$ and $c_-(K) =0$, by Corollary \ref{c*g*cor}. We do not know an example of a knot with $\mu(K)<\infty$ for which $\lceil m(K)\rceil <\mu(K)$, but expect that such knots exist; one possibility is the pretzel knot $P(-2,3,7)$ (see Example \ref{pretzelex}).

The paper is organized as follows. In Section \ref{S3proofsec} we assume Theorems \ref{mainthm} and \ref{scoopthmintro} and deduce the rest of the results stated above along with some additional refinements. Section \ref{equivcondsec} describes several conditions equivalent to the existence of a fillable positive surgery, including Theorem \ref{qpslicethm} giving a braid characterization. The latter is used to show that certain twist knots admit fillable positive surgeries, and leads to Table \ref{table:lowcross} in Section \ref{tablesec}, verifying Corollary \ref{tablecor}.

Section \ref{constructionsec} contains constructions and in particular gives the proof of Corollary \ref{existencecor}; it also includes some remarks on the question of whether any $L$-space knot admits a fillable positive surgery. 

In Section \ref{surgcoeffsec} we consider the invariant $\mu(K)$, the minimal fillable positive surgery coefficient. This section includes the proof of Propositions \ref{muboundprop} and \ref{torusknotprop}.

Section \ref{proofsec} contains the proofs of Theorem \ref{mainthm} and \ref{scoopthmintro} and is the technical core of the paper. Section \ref{tablesec} tabulates knots with up to 10 crossings that admit fillable positive surgeries.

{\bf Acknowledgments.} The authors wish to thank Peter Feller and Marco Golla for interesting discussions about the ideas in this paper; thanks also to Ken Baker and Brendan Owens for helpful communications. A portion of this work was carried out while the first author was a research member at the MSRI/SLMath program ``Analytic and Geometric Aspects of Gauge Theory;'' we thank them for their support and hospitality. TM was supported in part by grants from the Simons Foundation (numbers 523795 and 961391). BT was supported in part by grants from NSF (DMS-2105525 and CAREER DMS 2144363) and the Simons Foundation (636841, BT).

\section{Applications and Examples}\label{S3proofsec} For this section, we will assume the results of Theorems \ref{mainthm} and \ref{scoopthmintro}, and will mainly consider Legendrian and transverse knots in the standard contact structure on $S^3$. Before making this specialization we note that, given a knot with a fillable positive surgery, Theorem \ref{mainthm} does not specify directly {\it which} contact surgery will result in a fillable structure. However, we have the following basic observation. Let $Y$ be a 3-manifold and $K\subset Y$ a nullhomologous knot such that $(Y, K) = \partial (Z,\Delta)$ for $\Delta\subset Z$ a properly embedded disk in the 4-manifold $Z$. If one chooses a smooth 2-chain $\Sigma\subset Y$ with boundary $K$ (e.g., a Seifert surface), then framings on $K$ are identified with the integers by declaring the framing induced by $\Sigma$ to correspond to 0. Likewise, $\Sigma$ can be used to define a self-intersection number for $\Delta$ by defining $\Delta\cdot\Delta$ to be the self-intersection of the absolute homology class determined by $\Delta\cup -\Sigma$. (If $Y$ is a rational homology sphere, these definitions are independent of the choice of $\Sigma$.) The following is straightforward, or compare \cite[proof of Theorem 1.14]{CET}.

\begin{lemma}\label{fillsurgcoefflem}  In the situation above, let $Z'$ be the complement of a small tubular neighborhood of $\Delta$, and $Y' = \partial Z'$. Then $Y'$ is obtained from $Y$ by surgery along $K$ with coefficient equal to $-\Delta\cdot\Delta$, where the surgery coefficient and self-intersection are defined with respect to the same Seifert surface.\hfill$\Box$
\end{lemma}

In particular, if we are given $(Y,\xi)$ and Legendrian $\K\subset Y$, and are interested in the smallest positive contact surgery coefficient $n$ such that $\xi^-_n(\K)$ is fillable, we are equivalently interested in the greatest self-intersection number of a symplectic disk bounded by a transverse pushoff of $\K$ in some filling of $(Y, \xi)$. If that largest self-intersection number is $-p$, then the corresponding fillable contact surgery coefficient is $p - \tb(\K)$.

\begin{remark} In light of Lemma \ref{fillsurgcoefflem}, one consequence of Theorem \ref{scoopthmintro} is that if $\Delta\subset Z$ is a properly immersed disk having nullhomologous, positively transverse boundary, in a symplectic manifold weakly filling $(Y, \xi)$, then $-\Delta\cdot\Delta > \tb(\K)$ for any Legendrian approximation $\K$ of $K = \partial \Delta$. This is not hard to see directly: for example, one can attach a Stein 2-handle $H$ to $Z$ along $\K$. Then the Lagrangian core of $H$ admits a symplectic pushoff $\tDelta$ whose boundary is the transverse pushoff $K$, and with self-intersection $\tDelta\cdot\tDelta =\tb(\K) -1$. Gluing $\Delta$ and $\tDelta$ gives a symplectic sphere $S$ embedded in $Z\cup H$, where the latter is a weak filling of its boundary. By \cite{eliashberg:afewremarks}, $Z\cup H$ embeds in a closed symplectic 4-manifold that we can assume has $b^+_2 \geq 2$, and hence by \cite{mcduff90} we must have $S\cdot S < 0$. Since $S\cdot S = \Delta\cdot\Delta + \tDelta\cdot\tDelta$, it follows that $-\Delta\cdot\Delta \geq \tb(\K)$. Granted that the complement of $\Delta$ has fillable boundary given by some contact surgery on $\K$, the inequality must be strict since contact surgery with smooth framing equal to $\tb(\K)$ is overtwisted. (In fact, the same argument applies without the assumption that $K$ is nullhomologous, to show that the smooth surgery that yields $Y'$ in Theorem \ref{scoopthmintro} corresponds to a framing that is greater than the contact framing of $\K$.) 
\end{remark}

\begin{remark}\label{techremark} Strictly, having found a symplectic disk whose boundary is a transverse pushoff $K$ of $\K$, Theorem \ref{scoopthmintro} says we obtain a fillable positive surgery on a Legendrian approximation $\K'$ of $K$, which {\it a priori} is not the same Legendrian knot as the original $\K$. We claim that we can ignore this technicality, i.e. $\K$ and $\K'$ either both admit fillable positive surgeries or neither do. Indeed, since they are approximations of the same transverse knot, $\K$ and $\K'$ are related by negative stabilizations and destabilizations; we claim that a fillable positive contact surgery on $\K'$ is equivalent to some positive contact surgery on $\K$. To see this, recall first that if $\K$ is any Legendrian that has been stabilized, then $\xi_{+1}(\K)$ is overtwisted. Let us write $\K^{(k)}$ for the $k$-fold negative stabilization of $\K$. Then the fact that $\xi_n^-(\K) \cong \xi_{n+1}^-(\K^{(1)})$ for any $n>0$ (see \cite[Lemma 2.6]{LS:surgery})  shows that whenever $0<n\leq k$ we have $\xi_n^-(\K^{(k)}) = \xi_{+1}^-(\K^{(k-n+1)})$ is overtwisted, hence in particular is not weakly fillable. Now, to say that $\K$ and $\K'$ are related as above is equivalent to $\K^{(k)} = \K'^{(k')}$ for some $k$ and $k'$. If $\xi_n^-(\K')$ is fillable, then so is $\xi_{n+k'}^-(\K'^{(k')}) = \xi_{n+k'}^-(\K^{(k)})$ (being equivalent to $\xi_n^-(\K')$), and hence $n+k' > k$. In that case, we have $\xi_{n+k'}^-(\K^{(k)}) = \xi_{n+k'-k}^-(\K)$ is a fillable positive surgery on $\K$.
\end{remark}

\subsection{Equivalent conditions for fillable surgery}\label{equivcondsec} In this and following subsections, we consider Legendrian and transverse knots in $(S^3, \xi_{std})$. The symplectic fillings of the standard 3-sphere are classified by work of Gromov and McDuff \cite{gromov85,mcduff90}, and are all symplectomorphic to a symplectic blowup of the standard 4-ball. From Theorem \ref{mainthm} we get:

\begin{proposition}\label{blowupfillprop} A Legendrian $\K$ in $(S^3, \xi_{std})$ admits a fillable positive contact surgery if and only if a transverse pushoff of $\K$ bounds a symplectic disk embedded in a symplectic blowup of $B^4$. Moreover, any symplectic filling of such a contact surgery is symplectomorphic to the complement of a small neighborhood of such a disk.
\end{proposition}

Since a blowup of $B^4$ has negative definite intersection form, this immediately gives:

\begin{corollary}\label{negdefcor} If $\K$ is a Legendrian knot in $S^3$, and $n>0$ is such that $\xi_n^-(\K)$ is weakly symplectically fillable, then any weak symplectic filling $(Z, \omega)$ of $\xi_n^-(\K)$ has $b_2^+(Z) = 0$.
\end{corollary}

Since a disk $\Delta$ in a blowup of $B^4$ necessarily has non-positive self-intersection, it also follows from the above and Lemma \ref{fillsurgcoefflem} that for a Legendrian $\K\subset (S^3, \xi_{std})$, every fillable positive contact surgery along $\K$ has nonnegative smooth surgery coefficient. This observation is refined and extended by Proposition \ref{muboundprop}, to be proved in Section \ref{surgcoeffsec}.

The condition that $K$ bounds an embedded symplectic disk in a blowup of $B^4$ can be translated into the condition that $K$ bounds a singular symplectic disk in $B^4$ itself:

 \begin{theorem}\label{fillsurgS3thm} An oriented Legendrian $\K\subset S^3$ admits a fillable positive contact surgery if and only if the positive transverse pushoff of $\K$ is the boundary of a (possibly) singular symplectic disk in $B^4$ having locally holomorphic singularities.
 \end{theorem}
\begin{proof}
Let $Z \cong B^4\# k\cptwobar$ be a blowup of the 4-ball equipped with some symplectic structure $\omega$, obtained by blowing up the standard structure on $B^4$. Let $\Delta\subset Z$ be a properly embedded symplectic disk with boundary a positive transverse knot $K\subset S^3$. Adapting arguments of McDuff \cite[Lemma 3.1]{mcduff90} (see also Boileau-Orevkov \cite[proof of Lemma 2]{boilorev}) we can find an almost-complex structure $J$ on $Z$ compatible with $\omega$ such that $\Delta$ is $J$-holomorphic, and such that the exceptional spheres $E_1,\ldots, E_k$ of the blowup are also $J$-holomorphic: in particular each geometric intersection between $\Delta$ and $E_j$ contributes positively to the intersection number. By perturbing $\Delta$ slightly (symplectically) we can suppose that $\Delta$ intersects each $E_j$ transversely if we wish. In any case, the image of $\Delta$ under the blow-down $Z\to B^4$ is then a (possibly) singular symplectic disk whose singularities are modeled on those of a holomorphic curve. 
 
 Conversely, any singular symplectic disk as in the statement can be desingularized by suitable blowups, and remains symplectic in the blown-up 4-ball. 
 \end{proof}
   
 
Recall that a transverse knot in $(S^3, \xi_{std})$ is transversely isotopic to the closure of a braid; two braid closures represent transversely isotopic knots if and only if the closed braids are related by braid isotopy and by positive braid stabilization and destabilization. A braid $\beta$, as an element of the $n$-stranded braid group $\Br_n$, is {\it quasipositive} if it can be expressed as a product of conjugates of positive powers of the standard Artin generators $\sigma_1,\ldots, \sigma_{n-1}\in \Br_n$. We will say that a transverse knot is quasipositive if it is transversely isotopic to the closure of a quasipositive braid. The condition of quasipositivity is intimately tied to symplectic topology. Indeed, Boileau-Orevkov \cite{boilorev} show that if $K$ is the transverse boundary of a symplectic surface smoothly embedded in $B^4$ then $K$ is quasipositive. Conversely, by work of Rudolph \cite{rudolph:algebraic}, a quasipositive knot is the boundary of an algebraic curve (in particular, smooth symplectic surface) in $B^4$.

 Now, a singular symplectic disk can always be smoothed to an embedded symplectic surface of (in general) higher genus. From Theorem \ref{fillsurgS3thm} we infer the following slight refinement of Corollary \ref{c*g*cor}:
 
 \begin{corollary}\label{c*g*refined} If $\K$ is a Legendrian in $S^3$ that admits a fillable positive contact surgery, then the transverse pushoff $K$ of $\K$ is quasipositive. Moreover, $K$ bounds a smoothly and symplectically immersed disk in the 4-ball, having only positive double point singularities, and the number of double points is equal to the slice genus $g_*(K)$. 
 \end{corollary}

\begin{proof} The first claim follows from the discussion above. For the second, we make a symplectic perturbation of an embedded disk in a blowup of $B^4$ bounded by $K$ such that it intersects the exceptional spheres transversely. Observe that if for some exceptional sphere $E_j$ we have that $\Delta\cap E_j$ is transverse and consists of $m_j$ points, then the corresponding singularity after blowing down is an ordinary multiple point of order $m_j$.  Now make a further perturbation after the blowdown so that any higher-order multiple points become collections of ordinary double points. These are all positive double points because the singularity has a holomorphic model. To see there are exactly $g_*(K)$ double points, observe that we can find a nearby embedded symplectic surface in $B^4$ (smoothing the double points) with boundary $K$, whose genus equals the number of double points in the immersed disk. A symplectic surface bounding $K$ necessarily has genus equal to the slice genus \cite{boilorev, rudolph93}.
\end{proof}

\begin{remark}\label{doubleptrem} By the argument in the preceding proof, the condition in Theorem \ref{fillsurgS3thm} of bounding a singular symplectic disk with holomorphic singularities can be replaced by the condition that the symplectic disk have only ordinary double point singularities (all positive). In other words, the converse of Corollary \ref{c*g*refined} is true as well.
\end{remark}

In general, we expect that without the condition that the disk be symplectically immersed, the converse of Corollary \ref{c*g*refined} is false. However if we specialize to transverse knots whose underlying knot type is slice, i.e. bounds an embedded disk in $B^4$, then this condition is automatic:

\begin{corollary}\label{qpslicecor} A Legendrian representative $\K$ of a slice knot $K\subset S^3$ admits a fillable positive contact surgery if and only if a transverse pushoff of $\K$ is transversely isotopic to the closure of a quasipositive braid.
\end{corollary}

\begin{proof} If a transverse knot $K$ is slice and quasipositive, then a symplectic surface in $B^4$ bounded by $K$ must be a disk. In particular this provides the embedded symplectic disk in a filling of $S^3$ that yields a fillable surgery. The ``only if'' follows from Corollary \ref{c*g*refined}.
\end{proof}

\begin{corollary}\label{lagslicecor} Let $\K$ be a Legendrian knot such that:
\begin{itemize}
\item The underlying knot type of $\K$ is slice,
\item The Thurston-Bennequin invariant of $\K$ is $\tb(\K) = -1$. 
\end{itemize}
Then $\K$ is Lagrangian slice in $B^4$ if and only if the transverse pushoff of $\K$ is quasipositive as in Corollary \ref{qpslicecor}.
\end{corollary}
\begin{proof} A Lagrangian disk with Legendrian boundary can be perturbed to a symplectic disk with positively transverse boundary \cite{CGHS:lagfill,eliashberg:2knots}, hence by \cite{boilorev} again we find that if $\K$ is Lagrangian slice then the pushoff of $\K$ is quasipositive.

For the converse, quasipositivity means  that the transverse pushoff $K$ bounds an embedded symplectic (indeed complex) surface in $B^4$, which necessarily realizes the 4-ball genus of $K$. Since $K$ is slice we find that $K$ bounds an embedded symplectic disk $\Delta\subset B^4$, whose complement is then a symplectic filling of some positive contact surgery on $\K$ by Theorem \ref{scoopthmintro}. The complement of $\Delta$ in $B^4$ has the homology of $S^1\times B^3$, and its boundary the homology of $S^1\times S^2$. It follows that the topological surgery coefficient is necessarily 0, and since $\tb(\K)=  -1$ the contact surgery coefficient is $+1$. But by \cite{CET}, contact $+1$ surgery on $\K$ is fillable if and only if $\K$ is Lagrangian slice.
 \end{proof}
 
 The distinction between the situations of the two preceding  corollaries, both dealing with slice knots, can be viewed either in terms of the  surgery coefficient resulting in a fillable surgery, or in terms of the Thurston-Bennequin number. In either corollary, we find a fillable surgery by removing an embedded symplectic disk from the 4-ball; the resulting 3-manifold is necessarily obtained by smooth 0-surgery. The corresponding contact surgery coefficient depends on the Thurston-Bennequin number of $\K$; for example the mirror of the knot $8_{20}$ is slice, quasipositive, and admits a Legendrian representative having $\tb(\K) = -2$. The fillable contact manifold arising as the boundary of the complement of a symplectic slice disk for $\K$ is then smooth $0$-surgery, and contact $+2$ surgery on $\K$. Contact $+1$ surgery on this Legendrian knot is not fillable since $8_{20}$ is not Lagrangian slice \cite{CET}, corresponding to the fact that $8_{20}$ does not admit a Legendrian representative with Thurston-Bennequin number $-1$ (this also follows since the Heegaard Floer contact invariant of the corresponding contact structure vanishes, see \cite{MT:surgery}). 
 
\begin{remark} It appears to be an open question whether any slice knot that admits a Legendrian representative having $\tb = -1$ must be Lagrangian slice. In light of Corollary \ref{lagslicecor}, this is equivalent to the {\it a priori} weaker condition that the knot bound a symplectic disk in $B^4$, or equivalently that such a knot must have a quasipositive pushoff.
\end{remark}
 
 The condition for existence of a fillable positive surgery given in Theorem \ref{fillsurgS3thm} has an equivalent formulation in terms of quasipositive braid expressions. As mentioned previously, a braid $\beta\in\Br_n$ is said to be quasipositive if it can be expressed as a product of expressions of the form $w\sigma w^{-1}$, for $\sigma$ a standard generator of $\Br_n$ and $w\in \Br_n$ any element. Such an expression is called a {\it positive band}. In fact, any decomposition of $\beta$ as a product of positive bands can be thought of as arising from an algebraic curve,  specifically a smooth algebraic subset $C = \{F(z,w) = 0\}\subset D^2\times \cee \subset \cee\times \cee$, such that
 \begin{itemize}
 \item $F$ is a polynomial of the form $w^n +  f_{n-1}(z) w^{n-1} + \cdots +f_0(z)$. In particular, for fixed $z$ the equation $F(z,w) = 0$ has exactly $n$ solutions counted with multiplicity. It follows that if $D^2(R)$ denotes the disk of radius $R$ in $\cee$ (and we write $D^2$ for $D^2(1)$), then the curve $C$ is contained in $D^2 \times D^2(R)$ for sufficiently large $R$.
 \item The intersection $C\cap S^1\times D^2(R) \subset D^2\times D^2(R)$ is equivalent to the closure $\hat\beta$ of $\beta$. Here we smooth the corners of $D^2\times D^2(R)$ and identify the boundary with $S^3$, where the complex tangencies determine the standard contact structure.
\item The map $C\to D^2$ induced by the projection to $z$ is a simple branched covering, with branch points in 1-1 correspondence with the positive bands in the decomposition of $\beta$.
\item The braid word $\beta$ arises as the monodromy of $C$ around $\partial D^2$, in the sense described in \cite{rudolph:algebraic}.
\end{itemize}
The fact that any quasipositive braid closure arises this way was proved by Rudolph \cite{rudolph:algebraic}. We are interested in a variation on this situation, in which the curve $C$ is not necessarily smooth but may have nodes, and also is required to have genus 0. With this in mind, let us call an element of the braid group $\Br_n$ of the form $w \sigma^2 w^{-1}$ a {\it positive node}, where again $\sigma$ is a standard generator.

\begin{theorem}\label{qpslicethm} A Legendrian knot $\K$ in $(S^3, \xi_{std})$ admits a fillable positive contact surgery if and only if a transverse pushoff $K$ of $\K$ is transversely isotopic to the closure of a braid of the form
\[
\beta = \beta_1\cdots \beta_k\in \Br_n
\]
where each $\beta_j$ is either a positive band or a positive node, and such that there are exactly $n-1$ positive bands among the $\beta_j$.
\end{theorem}

It is not hard to see that the condition on the number of bands in the expression of the statement is equivalent to the requirement that there are exactly $g_*(K)$ positive nodes in that expression.

\begin{proof} By Theorem \ref{fillsurgS3thm} and Remark \ref{doubleptrem}, the claim is equivalent to the assertion that $K$ bounds an immersed symplectic disk in $B^4$ with only positive double points if and only if it is equivalent to a braid of the given form. The forward implication follows from an adaptation to the immersed case of Boileau-Orevkov's proof that a symplectic surface in $B^4$ is a quasipositive (or ``positively braided'') surface in the sense of \cite[Definition 1]{boilorev} (see also \cite{rudolph:special}). We note that the monodromy associated to a simple double point is the square of an Artin generator and thus each double point gives rise to a positive node factor in $\beta$. The remaining positive bands correspond to simple branch points of the covering $C\to D^2$;  the restriction on the number of positive bands is equivalent to the condition that the surface be irreducible and have genus 0, by an easy exercise with the Riemann-Hurwitz formula. 

In the other direction, a braid of the given form arises as the monodromy of an algebraic curve in $D^2\times D^2(R)$ as above, having a simple double point for each positive node $\beta_j$, as follows from Orevkov's generalization \cite{orevkov:realize} of Rudolph's construction (see also \cite[Theorem 73]{rudolph:handbook}). Granted that $\hat\beta$ is a knot, such a curve must have genus 0 and $k - n+1$ double points.
\end{proof}

The conditions in Theorem \ref{qpslicethm} and the preceding results allow for straightforward determination for many knots in the knot table whether or not there is a Legendrian representative admitting a fillable positive surgery. Indeed, by Corollary \ref{c*g*refined} we need consider only quasipositive knots, and for these it is often straightforward to obtain a braid expression of the sort given in Theorem \ref{qpslicethm}. For example, the knot $10_{142}$ is listed in KnotInfo \cite{knotinfo} as quasipositive with the following braid expression as  a product of positive bands (in which the only band that includes a nontrivial conjugation is enclosed in parentheses):
\[
\beta = \underbrace{\sigma_1\sigma_1}\sigma_1(\sigma_3\sigma_2\sigma_3^{-1})\underbrace{\sigma_1\sigma_1}\sigma_1\sigma_2\sigma_3\in \Br_4.
\]
To show $10_{142}$ admits a fillable positive surgery we must arrange this expression into a product of 3 positive bands, and the remaining terms must pair into positive nodes (conjugates of squares of generators). Two such nodes are indicated by the braces above, but this leaves 5 bands (the four remaining unconjugated generators and the band in parentheses) rather than 3. Instead, one can replace the occurrence of $\sigma_2$ in the parentheses by $\sigma_1\sigma_2\sigma_1\sigma_2^{-1}\sigma_1^{-1}$, and after a minor manipulation we obtain
\[
\beta = \underbrace{\sigma_1\sigma_1}\underbrace{\sigma_1\sigma_1}(\sigma_3\sigma_2\sigma_1\sigma_2^{-1}\sigma_3^{-1})\underbrace{\sigma_1\sigma_1}\sigma_2\sigma_3,
\]
containing only 3 bands and the remaining terms paired into nodes. Similar arguments, together with the conditions of quasipositivity and that $c_*(K) = g_*(K)$ from Corollary \ref{c*g*refined}, lead quickly to the determination for all knots of 10 or fewer crossings whether a fillable positive contact surgery exists; see Table \ref{table:lowcross} in Section \ref{tablesec}.

\begin{example} Similar direct arguments can be used to show that any positive twist knot admits a fillable contact surgery. Recall that the positive twist knot $K_{\ell}$ is described by the diagram on the left of Figure \ref{twistfig}, which can be rearranged into the 2-bridge position shown on the right of that figure. Note that $\ell$ is odd; we write $\ell = 2k+1$. A standard braiding algorithm, for example that described in \cite[Section 4]{baader05}, yields a description of $K_\ell$ as the closure of a quasipositive braid (indeed, strongly quasipositive) that includes a positive node term corresponding to the clasp in the original diagram.  Since the slice genus of $K_{\ell}$ is 1 the corresponding nodal symplectic surface must have genus 0, and we infer the existence of a fillable surgery by Theorem \ref{fillsurgS3thm}.
\begin{figure}[t]
\begin{center}
 \includegraphics[width=10cm]{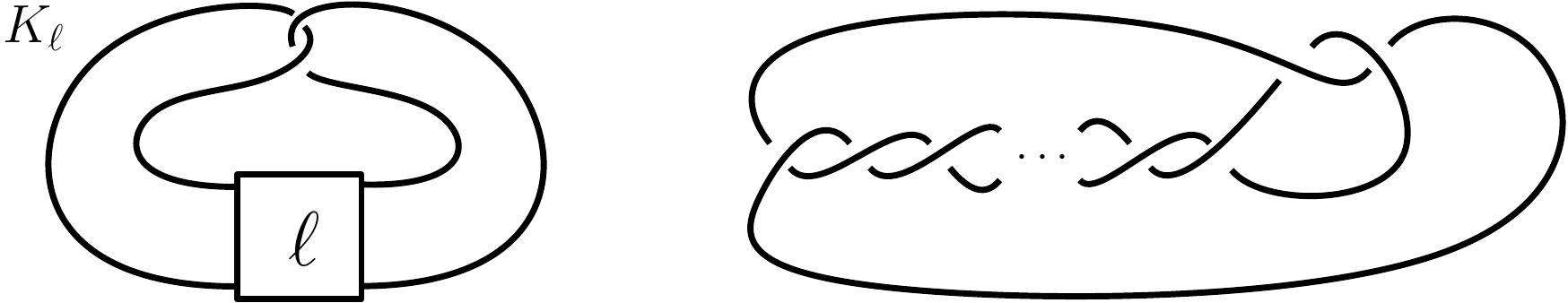}
 \caption{On the left is the positive twist knot $K_{\ell}$ where $\ell=2k+1$. On the right is $K_{\ell}$ in 2-bridge position where the lower twist region has $\ell$ half twists.}
  \label{twistfig}
\end{center}
\end{figure}

Explicitly, for indices $1\leq i\leq j\leq n-1$ we let $\sigma(i,j)\in \Br_n$ be the positive band given by the word
\[
\sigma(i,j) = \sigma_i \sigma_{i+1}\cdots \sigma_{j-1} \sigma_j \sigma_{j-1}^{-1} \cdots \sigma_{i+1}^{-1} \sigma_i^{-1}
\]
(such an expression is sometimes called an embedded positive band). Then one can show that $K_\ell$ is isotopic to the closure of the braid
\[
\beta = \sigma(k, 2k+1)\, \sigma(1, 2k+1)^2\,\prod_{i = 1}^k \sigma(i, k+i)\sigma(i+1, k+i) \in \Br_{2k+2}
\]
This expression contains $2k+1$ positive bands and one positive node, showing that $K_\ell$ bounds a symplectic disk with a single positive double point.
\end{example}

\subsection{Constructions}\label{constructionsec}

In this subsection we prove Corollary \ref{existencecor}, which shows that knots in certain families admit fillable positive surgeries. 

\begin{proof}[Proof of Corollary \ref{existencecor}, part (a)] We suppose that $K$ is a transverse knot in $(S^3,\xi_{std})$ given as the closure of a positive braid (that is, a braid expressed as a word in only positive powers of the standard generators of the braid group). Recall, for example from \cite[Lemma 2.4]{etnyregolla}, that an isotopy of transverse knots sweeps out a symplectic surface in the symplectization of the surrounding contact manifold; if the isotopy condition is relaxed to allow transverse self-intersections at isolated parameter values then the swept-out surface is symplectically immersed. Compared to \cite{etnyregolla} we run things backwards: we consider the ``negative half'' of the symplectization of $(S^3, \xi_{std})$ lying inside $B^4$, so increasing the time parameter in the isotopy corresponds to decreasing the radial coordinate.

In \cite[Section 4]{boilweb}, Boileau and Weber give an unknotting algorithm for the closure of any braid $\beta$ that requires at most  ${1\over 2}(\ell(\beta) - n + r)$ crossing changes, where $\ell(\beta)$ is the word length of $\beta$, $n$ is the braid index, and $r$ the number of components of the closure $\hat{\beta}$. In particular, each crossing change is made in the diagram of $\beta$ coming from its expression in braid generators. Applied to a positive braid whose closure is a knot, the algorithm provides a sequence of crossing changes, necessarily from positive to negative, from $\beta$ to a braid $\varepsilon$ that, after a braid isotopy, contains exactly $n-1$ (positive) crossings and whose closure is unknotted. It is easy to see that the closure $\hat\varepsilon$ of such a braid bounds a smooth symplectic disk in $B^4$. (For example, the self-linking number of $\hat{\varepsilon}$ is the writhe, or algebraic word length, minus the braid index, in this case $-1$. Since $\hat{\varepsilon}$ is unknotted and the unknot is transversely simple, $\hat\varepsilon$ is transversely isotopic to a standard unknot. Alternatively, one can use the results cited in the proof of Theorem \ref{qpslicethm} to construct a smooth symplectic disk bounded by $\hat\varepsilon$ without involving any braid (de)stabilizations.) This disk together with the immersed annulus corresponding to the sequence of crossing changes gives an immersed symplectic disk in $B^4$ with boundary  $\hat\beta$; all self-intersections are positive since all crossings changed from positive to negative. Hence $K$ has a fillable positive surgery by Theorem \ref{fillsurgS3thm}.
\end{proof}

\begin{proof}[Proof of Corollary \ref{existencecor}, part (b)] It is well known that every tight contact structure on a lens space is symplectically fillable \cite{Honda:classification1}. Hence, given a lens space knot $K\subset S^3$ (which is to say a smooth knot such that $S^3_m(K)$ is a lens space for some $m>0$), it suffices to show that there is a Legendrian representative of $K$ so that the contact surgery corresponding to $m$ has positive contact surgery coefficient and is tight. Now, a lens space knot is necessarily strongly quasipositive (see \cite{hedden:positivity, ni:fibered, OS:lens}, or \cite{baldwinsivek:sqp} for a self-contained proof). In particular $K$ is isotopic to the boundary of an embedded symplectic surface $S$ in $B^4$, and is thus can be realized as a transverse knot in the standard contact structure of $S^3$ such that the self-linking number $s\ell(K)$ is equal to $2g(S) -1$, and $g(S) = g_*(K)$ (which equals $g(K)$ by strong quasipositivity). A Legendrian approximation $\K$ of this transverse knot then has $\tb(\K) + |\rot(\K)| = 2g_*(K)-1$. If $\K$ is oriented suitably, it then follows from \cite{MT:surgery, Golla} that   $\xi^-_n(\K)$ has nonvanishing Heegaard Floer contact class---and in particular is tight---whenever the smooth surgery coefficient $m = n + \tb(\K)$ is at least $2g_*(K)$. (Note that this conclusion requires the additional hypothesis that $\varepsilon(K)\geq 0$, where $\varepsilon$ is the concordance invariant defined by Hom \cite{hom:cables}. That this holds follows from \cite[Proposition 3.6(4)]{hom:cables}, using the fact that for strongly quasipositive knots, $\tau(K) = g_*(K) = g(K)$.) 

Hence for tightness we need to see that the surgery coefficient yielding a lens space (corresponds to a positive contact surgery and) is at least $2g_*(K)$. But by \cite[Theorem 1.1]{greene:Lspacesurg}, if surgery on $K$ with coefficient $m>0$ is a lens space then $m> 2g(K)-1$, which is also at least $\tb(\K)$. Thus the contact surgery giving a lens space has positive contact surgery coefficient and is tight, hence fillable. \end{proof}

For part (c) of Corollary \ref{existencecor}, we need a definition. Recall that a {\it pattern} (for a satellite operation) is a knot $P\subset S^1\times D^2$ in a solid torus. For a knot $K\subset S^3$, the satellite knot $P(K)$ is the (isotopy class of the) image of $P$ under an embedding $S^1\times D^2\to S^3$ that maps $S^1\times\{0\}$ onto $K$, and a pushoff $S^1\times z$ to a $0$-framed longitude of $K$. More generally, the twisted satellite $P_\ell(K)$ is obtained by a similar construction with the condition that $S^1 \times z$ is identified with an $\ell$-framed longitude, or equivalently links $K$ $\ell$ times. 

\begin{definition}\label{braidpatterndef} Consider $S^3$ as the unit sphere in $\cee^2$, and fix a smooth identification $S^1 \times D^2\to S^3 \cap \{|z_2|\leq {1\over 2}\}$ as a 0-framed neighborhood of the unknot in $S^3$. A {\em braided fillable pattern} is a pattern knot $P\subset S^1\times D^2$, subject to the following conditions:
\begin{itemize}
\item Under the identification above, $P$ is a transverse knot in $S^1\times D^2\subset S^3$.
\item There exists an immersed symplectic disk $\Delta_P \subset \{|z_2|\leq {1\over 2}\}\subset B^4$ with only positive double point singularities, with $\partial \Delta_P = P$.
\end{itemize}
\end{definition}

Clearly a braided fillable pattern, considered as a knot in $S^3$, is presented as a braid closure: that is, $P(U) = \hat\beta$ for some braid $\beta$, where $U$ is the unknot. Moreover, $\beta$ can be expressed as a product of positive bands and positive nodes, satisfying the condition in Theorem \ref{qpslicethm}. On the other hand if $\beta$ is expressed as in Theorem \ref{qpslicethm}, then as in the proof of that theorem, $P(U)$ is identified with the boundary of a nodal algebraic curve in $D^2\times D^2(R)$. Rescaling the second coordinate if necessary, we can find an algebraic, hence symplectic, disk as specified in the definition above. Part (c) of Corollary \ref{existencecor} is equivalent to the following.

\begin{proposition} Let $\C$ be a Legendrian in $S^3$ that admits a fillable positive contact surgery, whose corresponding smooth surgery coefficient is $m$. Let $C$ be the smooth knot underlying $\C$, and let $P$ be a braided fillable pattern. Then $K = P_m(C)$ admits a Legendrian representative having a fillable positive contact surgery. 
\end{proposition}

\begin{proof} We wish to show that $K$ bounds an embedded symplectic disk in some blowup of $B^4$. By hypothesis, $C$ can be described as the boundary of such a disk $\Delta_C\subset B^4\# k\cptwobar$, such that $-\Delta_C\cdot \Delta_C = m$. Let $\Delta_P\subset B^4$ be a nodal algebraic disk lying in $\{|z_2|\leq {1\over 2}\}$, whose intersection with $S^3$ is the pattern $P$. By further rescaling $z_2$ as above, we can suppose that $\Delta_P$ lies in $\{|z_2| < \epsilon\}$ for any chosen $\epsilon>0$; again rescaling the second coordinate preserves the fact that the immersed disk is holomorphic and hence it also remains symplectic.

For convenience, we can rescale the symplectic form on $B^4\# k\cptwobar$ so that the disk $\Delta_C$ has area $\pi$. Then by standard neighborhood theorems in symplectic geometry, for small enough $\epsilon >0$ there is a symplectomorphism $F$ from $D^2 \times D^2(\epsilon)$ to a neighborhood of $\Delta_C$. Strictly, this may involve ``trimming'' a subset of $D^2\times D^2(\epsilon)$ contained in $\{1-\delta \leq z_1 \leq 1\}\times D^2(\epsilon)$, but this does not affect the argument. In particular, the image under $F$ of the disk $\Delta_P$ is a singular symplectic disk in $B^4 \# k\cptwobar$, having only positive nodes as singularities, whose boundary is clearly the satellite $P_\ell(C)$ for some $\ell$. Moreover, if we take $C$ to be positively transverse (i.e., a transverse pushoff of $\C$), then since $P$ is braided we can also choose neighborhoods small enough that the satellite $P_\ell(C)$ is transverse.

The disk $F(\Delta_P)$ can be smoothed by further blowups, so it follows from Proposition \ref{blowupfillprop} that (a Legendrian approximation of) $P_\ell(C)$ has a fillable positive contact surgery. It remains to observe that the framing $\ell$ is nothing but the difference between the $0$-framing of $C$ and the framing that extends over the disk $\Delta_C = F(D^2 \times 0)$; this is exactly $m$.
\end{proof}

\begin{proof}[Proof of Corollary \ref{existencecor}, part (d)] Recall that if $C\subset S^3$ is a knot, the $(p,q)$ cable $C_{p,q}$ is the satellite knot having companion $C$ and pattern $P = T(p,q)$ the $(p,q)$ torus knot, thought of as a closed braid in the standard solid torus in $S^3$. Here $p$ is the ``longitudinal'' coordinate, and we consider $p, q >0$. By realizing a positive torus knot as the closure of a positive braid, we have that $T(p,q)$ is a braided fillable pattern in the sense of Definition \ref{braidpatterndef}. 

Observe that for a framing coefficient $m$, the twisted cable $P_m(C)$ (with $P = T(p',q')$) is nothing but $C_{p', q' + mp'}$. Put another way, we have $C_{p,q} = P_m(C)$ where $P = T(p, q-mp)$. According to part (c) of Corollary \ref{existencecor}, we infer that $C_{p,q}$ admits a fillable positive surgery so long as $m$ is the smooth coefficient of a fillable surgery on $C$, and $T(p, q-mp)$ is braided fillable. The second of these simply requires $q-mp >0$, which is the condition assumed in the statement.
\end{proof}

With the results of Corollary \ref{existencecor} in hand it is worth pausing for a few additional remarks. As mentioned in the introduction, the fact that a knot that admits a lens space surgery also has a fillable positive contact surgery implies the purely topological consequence that for lens space knots $K$ one has $g_*(K) = c_*(K)$. We do not know another proof of this fact.

It is natural to wonder whether the same conclusion holds for general $L$-space knots (i.e., knots $K$ such that some positive surgery along $K$ yields a Heegaard Floer $L$-space), in particular whether $c_* = g_*$ and the other consequences of Corollary \ref{c*g*cor} hold for such knots. There are certain parallels between the results of Corollary \ref{existencecor} and known constructions of $L$-space knots: for example, it is a consequence of work of Hedden \cite{hedden:cabling2} and Hom \cite{hom:cabling} that a cabled knot $K_{p,q}$ is an $L$-space knot if and only if $K$ is an $L$-space knot and ${q\over p}\geq 2g(K)-1$. (Note that since $L$-space knots are strongly quasipositive, we have $g(K) = g_*(K)$ for such knots.) From Corollary \ref{existencecor} we find that if $K$ admits a fillable positive surgery, then $K_{p,q}$ does as well, at least for ${q\over p}$ large enough; by Proposition \ref{muboundprop} (proved in the next section), it suffices that ${q\over p} \geq 4g_*(K)$. Of course one wonders about the values of $q\over p$ between $2g_*-1$ and $4g_*$. 

Likewise, there are more general satellite constructions of $L$-space knots \cite{HLV2014,hom:satellite}, requiring conditions on the pattern and companion knots as well as on the twisting parameter, analogous to the conditions in Corollary \ref{existencecor}(c) but somewhat less stringent; in particular it is not hard to construct $L$-space knots with these techniques, for which Corollary \ref{existencecor} does not guarantee a fillable positive surgery (see below).

On the other hand, the sufficient condition on the cabling parameter required to obtain a cabled knot with fillable surgery, stated in Corollary \ref{existencecor}(d), is certainly not necessary in general. For example, if one considers cables of the right-handed trefoil knot $T(2,3)$, it it easy to see that the cable $T(2,3)_{2,7}$ is isotopic to the closure of a positive braid, hence admits a fillable positive surgery. In this case the parameter ${q\over p}= {7\over 2}$ is less than the minimum smooth coefficient yielding a fillable surgery on $T(2,3)$ itself, which is 4.

The results of Corollary \ref{existencecor} combined with those of \cite{hedden:cabling2} provide many examples of $L$-space knots (that are not lens space knots) that admit a fillable positive surgery. It is natural to ask whether {\it every} $L$-space knot admits a fillable positive surgery; as a concrete instance we do not currently know whether the cable knot $T(2,3)_{2,3}$, which is an $L$-space knot but neither a lens space knot \cite{gordon:satellite} nor a positive braid closure \cite[Remark 4]{ito:homfly}, \cite[Example 1]{bakerexp}, admits a fillable positive surgery.
 
\subsection{Fillable surgery coefficients}\label{surgcoeffsec}

We now return to the question of which contact surgery on a given Legendrian $\K$ in the standard 3-sphere will yield a fillable contact structure. According to Proposition \ref{blowupfillprop} and Lemma \ref{fillsurgcoefflem} such a surgery exists if and only if the transverse pushoff $K$ bounds an embedded symplectic disk $\Delta$ in $Z = B^4\# k\cptwobar$, and in this case the fillable surgery corresponds to smooth surgery coefficient $r = -\Delta\cdot \Delta$. Here since $\partial Z = S^3$ the surgery coefficient and self-intersection do not depend on a choice of Seifert surface for $K$, so we make no further mention of that choice and in fact identify the homology class $[\Delta]$ with an absolute homology class via the natural isomorphism $H_2(Z)\to H_2(Z, \partial Z)$. With this understood, we can write 
\begin{equation}\label{diskclasseqn}
[\Delta] = \sum_{j=1}^k n_j e_j
\end{equation}
for integers $\{n_j\}$,where $e_1,\ldots, e_k$ are the homology classes of the hyperplanes in each copy of $\cptwobar$. In this case, since $e_i \cdot e_j = -\delta_{ij}$, we see that $\Delta$ determines a fillable contact surgery on $\K$ whose corresponding smooth coefficient is
\begin{equation}\label{sumsquareseq}
r = -\Delta\cdot\Delta = \sum_j n_j^2.
\end{equation}

Recall that we can arrange that each class $e_j$ is represented by a smooth symplectic 2-sphere $E_j$ having self-intersection $-1$, such that each point of intersection between $\Delta$ and $E_j$ contributes positively to $\Delta\cdot e_j$. Hence in the expression \eqref{diskclasseqn},  we have that $n_j = 1$ if and only if $\Delta$ intersects $E_j$ transversely in a single point. In particular, in this case we can blow down $E_j$ and the image of $\Delta$ after blowing down is still an embedded symplectic disk with boundary $K$, now in $B^4\#(k-1)\cptwobar$. Moreover, the topological surgery coefficient $-\Delta\cdot\Delta$ decreases by 1 after this operation. (Of course, $n_j = 0$ if and only if $\Delta$ is disjoint from $E_j$, and we can blow down $E_j$ without affecting the surgery.) Conversely, by blowing up a smooth point of $\Delta$ we decrease the self-intersection by 1 and correspondingly increase the smooth coefficient of a fillable surgery by 1.

One consequence of these remarks is the following result concerning the surgery coefficient of a fillable surgery.

\begin{proposition}\label{4gprop} If $\K$ is a Legendrian knot in $S^3$ such that some positive contact surgery is fillable, then $\xi^-_n(\K)$ is weakly fillable for every $n$ corresponding to a smooth surgery coefficient greater than or equal to $4g_*(K)$, where $g_*(K)$ is the slice genus of the underlying smooth knot $K$.
\end{proposition}

\begin{proof} As in the discussion surrounding Theorem \ref{fillsurgS3thm}, we can find a symplectic disk immersed in $B^4$ having boundary a pushoff $K$ of $\K$, only positive transverse double points as singularities, and exactly $g_*(K)$ such double points. By blowing up each double point we find an embedded symplectic disk $\Delta$ in the $g_*(K)$-fold blowup  of $B^4$, such that 
\[
[\Delta] = -\sum_{j=1}^{g_*(K)} 2e_j.
\]
In particular we have $\Delta\cdot\Delta = -4g_*(K)$, so that contact surgery on $\K$ with smooth coefficient $4g_*(K)$ is fillable.  It follows from the remarks preceding the statement of the proposition that every contact surgery with greater coefficient is also fillable.\end{proof}

In fact, if a disk $\Delta\subset B^4\# k\cptwobar$ is found, with transverse boundary $K$ and homology class as in \eqref{diskclasseqn}, the coefficients $n_j$ are related to the slice genus $g_*(K)$ as follows. Arrange as above that each intersection of $\Delta$ with an exceptional sphere is transverse, so that blowing down results in a disk having $k$ multiple points of orders $n_1,\ldots, n_k$. An ordinary multiple point of order $n_j$ can be perturbed symplectically to a collection of ${1\over 2}n_j(n_j-1)$ ordinary double points, which can each then be smoothed to give an embedded symplectic surface in $B^4$. Since this surface must realize the slice genus of $K$, we have:

\begin{proposition}\label{genusformulaprop} If a transverse knot $K$ bounds an embedded symplectic disk $\Delta\subset B^4\# k\cptwobar$ with homology class as in \eqref{diskclasseqn}, then
\[
2g_*(K) = \sum_{j=1}^k n_j (n_j-1).
\]
In particular, the complement of a neighborhood of $\Delta$ gives a symplectic filling of the result of a positive contact surgery on a Legendrian approximation of $K$, whose smooth surgery coefficient $r$ is given by
\[
r = 2g_*(K) + \sum_{j=1}^k n_j.
\]
\end{proposition}
The last equation follows from \eqref{sumsquareseq}. Since all $n_j >0$ we infer the following, which together with Proposition \ref{4gprop} completes the proof of Proposition \ref{muboundprop}.

\begin{corollary} Any fillable contact surgery on a Legendrian knot $\K$ with underlying smooth knot $K$ has smooth surgery coefficient $r$ satisfying $r\geq 2g_*(K)$. The inequality is strict unless $K$ is isotopic to the transverse boundary of an embedded symplectic disk in $B^4$, in particular unless $g_*(K) = 0$.
\end{corollary}

In a related vein, we have the following simple observation.

\begin{proposition}\label{fillprop} Let  $\K$ be a Legendrian in $S^3$, and suppose $n>0$ is an integer such that
\begin{itemize}
\item The contact structure $\xi_n^-(\K)$ is weakly fillable.
\item The smooth surgery coefficient $r = n+\tb(\K)$ satisfies
\[
r\in\{1,2 , 3,5, 6, 7, 10, 11,14, 15, 19\}.
\]
\end{itemize}
Then $\xi^-_{n-1}(\K)$ is also weakly fillable.
\end{proposition}

\begin{proof} The significance of the list of allowed values of $r$ is that these are exactly the integers such that any expression of $r$ as a sum of squares as in \eqref{sumsquareseq} must have at least one of the  $n_j$ equal to 1. We have seen that in this situation, the corresponding embedded disk in the blown-up 4-ball intersects one of the exceptional spheres transversely in a single point, so we can blow down that sphere and obtain a filling of the next smaller integer surgery.
\end{proof}

It follows from this result, for example, that for a non-slice knot $K$, the smallest possible smooth surgery coefficient corresponding to a fillable contact surgery is $4$.

Proposition \ref{4gprop} shows that the smooth surgery coefficient $4g_*(K)$ is always large enough to realize a fillable contact surgery, if any such surgery exists; moreover this coefficient corresponds, in a sense, to a generically immersed symplectic disk in $B^4$ bounding $K$. It is an interesting problem to determine, for a given $K$, what is the smallest coefficient of a fillable surgery; values smaller than $4g_*(K)$ correspond to symplectic disks with more ``interesting'' (e.g., higher multiplicity) singularities than transverse double points. 

\begin{example}\label{pretzelex} Consider the pretzel knot $K = P(-2,3,7)$. It was observed by Fintushel and Stern \cite{FS:lens} that $18$-surgery along $K$ is a lens space, and therefore by Corollary \ref{existencecor}(b) a suitable Legendrian representative of $K$ has a corresponding fillable contact surgery. Since the slice genus of $K$ is 5, Proposition \ref{4gprop} implies that the contact surgery yielding smooth $20$-surgery is fillable; of course the lens space surgery with coefficient 18 is also fillable. We claim that in fact, the smallest fillable contact surgery on the chosen representative corresponds to smooth 17-surgery. The observation that facilitates this result is that since $K$ has slice genus 5, any properly embedded symplectic surface in $B^4$ having boundary $K$ is genus 5. 

As one approach to see that 17-surgery must be fillable, first note that the fillable 18-surgery corresponds to an embedded symplectic disk $\Delta\subset B^4 \# k\cptwobar$ with $-\Delta\cdot\Delta = 18$. Observe that the only expression of 18 as a sum of squares not including 1 is $3^2 + 3^2$. A symplectic disk $\Delta$ realizing this as in \eqref{diskclasseqn} and \eqref{sumsquareseq} will have $k = 2$ and $n_1 = n_2 = 3$, and it then would follow from Proposition \ref{genusformulaprop} that $g_*(K)= 6$, a contradiction. Hence $\Delta$ must have intersection number 1 with at least one exceptional curve, which as before implies that the next smaller surgery coefficient is fillable.

Now we claim that the contact surgery corresponding to smooth 16-surgery  on $K$ is {\it not} fillable. Indeed, suppose there is a disk $\Delta\subset B^4 \# k\cptwobar$ with $-\Delta\cdot\Delta = 16$, with homology class as in \eqref{diskclasseqn}. If some $n_j = 1$, then we can blow down and obtain a filling of smooth 15-surgery, which is negative definite by Corollary \ref{negdefcor}. But according to \cite[Theorem 1.3]{KT:nonfill}, the manifold obtained by 15-surgery on $P(-2,3,7)$ does not admit any fillable contact structures; in fact it does not bound any negative definite 4-manifold. Hence we must have that \eqref{sumsquareseq} expresses 16 as a sum of squares not including 1, which means either $k = 4$ with $n_1 =\cdots = n_4 = 2$, or $k = 1$ with $n_1 = 4$. In the first case, an application of Proposition \ref{genusformulaprop} forces $g_*(K)=4$, while the second case yields genus 6, neither of which are possible.

Note that the fillability of smooth 17-surgery can also be seen directly from the expression of $P(-2,3,7)$ as the closure of the positive braid $\sigma_1\sigma_2^2\sigma_1^2\sigma_2^7\in\Br_3$ as listed in KnotInfo. Indeed, after a conjugation, application of a braid relation, and introduction of the factor $\sigma_2\sigma_2^{-1}$, the closure of the given braid is equivalent to the closure of the quasipositive braid 
\[
(\sigma_1\sigma_2)^3(\sigma_2^{-1}\sigma_1\sigma_2)\sigma_2^2\sigma_2^2\sigma_2,
\]
which we can see as the braid monodromy of a singular symplectic disk having two nodes, along with a cusp singularity corresponding to the full twist $(\sigma_1\sigma_2)^3$. The latter becomes smooth after a single blowup of multiplicity three while the two nodes are each multiplicity two, yielding an embedded symplectic disk $\Delta\subset B^4 \# 3\cptwobar$ with $[\Delta] = -3e_1 -2 e_2 - 2 e_3$ and self-intersection $-17$.

\end{example}

We now consider the problem of determining the minimal coefficient for a fillable surgery in the case of a positive torus knot $T(p,q)$, realized as a transverse knot and the closure of a positive braid in the standard way. In our notation we will always take $p$ and $q$ relatively prime, and since $T(p,q)$ is transversely isotopic to $T(q,p)$, we will also assume that $p>q$.

To begin, recall from Corollary \ref{negdefcor} that for any knot $K\subset S^3$ admitting a fillable positive contact surgery corresponding to smooth surgery coefficient $n$ (which is also necessarily positive), a symplectic filling of the corresponding contact structure is negative definite: in particular the surgery manifold $S^3_n(K)$ is the boundary of some smooth, compact, negative definite 4-manifold. As mentioned in the introduction, Owens and Strle \cite{OwSt2012} define 
\[
m(K) = \inf\{ r\in \cue_{>0}\,|\, \mbox{$S^3_r(K)$ bounds a negative definite 4-manifold}\}.
\]
For torus knots they determine $m$ explicitly, in fact they show
\begin{equation}\label{torusknotm}
m(T(p,q)) = pq - c(p,q)
\end{equation}
for a certain rational number $c(p,q)$ depending on the continued fraction expansion of $p/q$. (Moreover, though $m(T(p,q))$ is defined as an infimum, Owens-Strle show that the manifold $S_{m(T(p,q))}(T(p,q))$ bounds a negative definite 4-manifold.) Here the ``Euclidean'' continued fraction expansion is ${p\over q} = [a_1,\ldots, a_n]^+$ if we have
\begin{equation}\label{contfrac}
{p\over q} = a_1 + {1 \over {a_2 + {1\over \cdots +{1\over a_{n-1} + {1\over a_n}}}}}.
\end{equation}
Any rational number ${p\over q}>1$ can be written in this form with $a_j\geq 1$ for all $j$, and the expression is unique if we require that $a_n\geq 2$ (we will always assume these conditions). With this in mind, the constant $c(p,q)$ in \eqref{torusknotm} is given by
\[
c(p,q) = \left\{ \begin{array}{ll}\ds\frac{q}{p^*} & \mbox{if $n$ is even in \eqref{contfrac}} \vspace{1ex}\\ \ds\frac{p}{q^*} & \mbox{if $n$ is odd.}\end{array}\right.
\]
Here $p^*\in \{1,\ldots, q-1\}$ and $q^*\in \{1,\ldots, p-1\}$ denote the inverses of $p$ and $q$ modulo $q$ and $p$, respectively.  The proof of \cite[Proposition 4.1]{OwSt2012} also shows that
\[
c(p,q) = [a_n, a_{n-1}, \ldots, a_m]^+
\]
where $m$ is 2 or 1 when $n$ is even or odd, respectively. This implies that the greatest integer less than or equal to $c(p,q)$ is  $a_n$, the last coefficient in the continued fraction for $p/q$. In particular the smallest positive {\it integer} surgery on $T(p,q)$ that bounds a smooth negative definite 4-manifold (symplectic or not) is 
\[
\lceil m(T(p,q)) \rceil = pq - \lfloor c(p,q)\rfloor = pq - a_n.
\]

\begin{theorem}\label{torusknotthm} Let $\T(p,q)$ be a Legendrian knot in the knot type of $T(p,q)$, having maximal Thurston-Bennequin invariant. Then $\T(p,q)$ admits a fillable positive contact surgery, and the smallest integral such surgery is the one corresponding to smooth surgery coefficient $\lceil m(T(p,q)) \rceil$.
\end{theorem}

\begin{proof} From Corollary \ref{negdefcor} and the results of Owens-Strle, the smallest fillable contact surgery on $\T(p,q)$ has smooth coefficient at least $\lceil m(T(p,q)) \rceil$. We wish to show that this smallest possible contact surgery is indeed symplectically fillable; in view of the remarks at the beginning of this section this is equivalent to showing that $T(p,q)$ bounds an embedded symplectic disk $\Delta$ in a blowup of $B^4$, such that $-\Delta\cdot \Delta = pq - a_n$.

To do so, we recall some basic facts about algebraic curves and blowing up. The torus knot $T(p,q)$ (being a transverse pushoff of $\T(p,q)$ and a positive braid) can be realized as the link of the singularity at the origin of the curve $C = \{x^p + y^q = 0\}$ in $\cee^2$. The blowup $\pi: \widetilde \cee^2 \to \cee^2$ at the origin can be described in coordinates by the transformation $x = u$, $y = uv$ (or, in another chart, by $x = uv$, $y = u$). Transforming the equation for $C$ by this rule gives the total transform $\pi^*C$, namely
\[
\pi^* C = \{u^q( u^{p-q} + v^q) = 0 \}.
\]
The reduced curve $u = 0$ describes the exceptional curve of the blowup in this coordinate system, and the remaining expression $u^{p-q} + v^q = 0$ corresponds to the proper transform $\widetilde{C}$ of $C$. In particular, from the expression above we see that the exceptional curve $E$ appears with multiplicity $q$ in $\pi^*C$, and we say that $q$ is the multiplicity of the singularity. It is standard to write $\widetilde{C} = C - q E$ as divisors where, by abuse of notation, $C$ is conflated with $\pi^*(C)$. (Note that in the other chart on the blowup, the proper transform is smooth and disjoint from the exceptional curve.) It follows from this that the self-intersection of $\widetilde{C}$ is $q^2$ less than that of $C$. Of course, the original curve $C$ lies in the 4-ball, hence its homology class is trivial and the self-intersection vanishes; we phrase the result as we do because of the iteration to follow.

Having recorded this effect on the self-intersection of the (in general, still singular) proper transform $\widetilde{C}$, note that $\widetilde{C}$ is a curve described by the same sort of equation as $C$, and we can iterate this procedure. Specifically, if $p-q > q$ then thinking of $\widetilde{C}$ as given by the equation $x^{p-q} + y^q = 0$, a blowup identical to the first one will give another exceptional curve of multiplicity $q$ and decrease the self-intersection by another $q^2$. On the other hand, if $p-q = r < q$, then we blow up the curve $x^r + y^q = 0$ using the transformation $x = uv$, $y = u$, and find the total transform $u^r(v^r + u^{q-r}) = 0$. Thus the blowup has multiplicity $r$ and the self-intersection decreases by $r^2$.

Clearly, this procedure is closely tied to the Euclidean algorithm for the quotient $p/q$, and thus its continued fraction. (See, for example \cite[Proposition 6.11]{pp2007}.) For ease of notation, let us write $p = p_0$ and $q = p_1$. Then the Euclidean algorithm yields:
\begin{eqnarray*}
p_0 &=& a_1 p_1 + p_2 \\
p_1 &=& a_2 p_2 + p_3 \\
 & \vdots & \\
 p_{n-2} &=& a_{n-1} p_{n-1} + p_n\\
 p_{n-1} &=& a_n p_n
 \end{eqnarray*}
 where in each line we consider dividing $p_j$ by $p_{j+1}$, with quotient $a_{j+1}$ and remainder $p_{j+2} < p_{j+1}$. The algorithm terminates when there is no remainder as in the last line above, and then $p_n$ is the greatest common divisor of $p_0$ and $p_1$. In our case of course this means $p_n=1$. Moreover, the coefficients $a_j$ are simply those appearing in \eqref{contfrac}.
 
With the earlier discussion in mind, the first line of the algorithm tells us that in resolving the singularity of $C$ we perform $a_1$ blowups each of multiplicity $p_1$, and hence each decreasing the self-intersection by $p_1^2$. Similarly the $j$-th line means we use $a_j$ blowups, each of multiplicity $p_j$. The final line suggests that $a_n$ blowups are required, but observe that since $p_n = 1$ the curve obtained in the previous step is of the form $0= x^{p_{n-1}} + y^{p_n} = x^{p_{n-1}} + y$, which is already smooth (it does, of course, have a tangency of order $a_n$ with the newest exceptional curve, but this is not relevant for us). 

The result of these considerations is that $T(p,q)$ bounds a smooth, properly embedded and complex (hence symplectic) disk $\Delta$ in a blowup of $B^4$ whose self-intersection is given by
\[
-\Delta\cdot \Delta = a_1 p_1^2 + a_2 p_2^2 + \cdots + a_{n-1}p_{n-1}^2.
\]
On the other hand, if for $j = 1,\ldots, n$ we multiply the $j$-th line of the Euclidean algorithm above by $p_j$, then sequentially substitute each line into the previous, we find
\[
p_0p_1 = a_1 p_1^2 + a_2 p_2^2 + \cdots + a_{n-1}p_{n-1}^2 + a_n p_n^2.
\]
Recalling that $p_n = 1$ and reverting to $p_0 =p$ and $p_1 = q$, this says $pq - a_n = -\Delta\cdot \Delta$ as desired.\end{proof}

\begin{remark} From this result it is natural to expect that if one considers positive rational contact surgeries on $\T(p,q)$, the smallest such surgery that is fillable will correspond to smooth coefficient exactly the rational number $m(T(p,q))$.
\end{remark}

\section{Proof of Main Theorems}\label{proofsec}

While Theorem \ref{mainthm} is stated in terms of contact surgery along Legendrian knots, for our purposes it is more natural to work with transverse knots and transverse surgery. We will review the relationship between these ideas (which was clarified in particular by the work of Conway \cite{Conway:transverse}) below.

\subsection{Open books and contact and transverse surgery}
Recall that an open book decomposition of a closed, connected, oriented 3-manifold $Y$ consists of an oriented link $B\subset Y$, the {\it binding} of the open book, together with a fiber bundle projection $\pi: Y - B \to S^1$ such that for each component $B_i$ of $B$, there is a tubular neighborhood $nbd(B_i) \cong  D^2\times B_i$  in which $\pi$ is identified with projection on the angular coordinate of $D^2$. These neighborhoods must be suitably compatible with orientations, so that for each $\theta\in S^1$ surface ${\Sigma}_\theta = \overline{\pi^{-1}(\theta)}$ given by the closure of  the preimage of $\theta$ is an oriented, connected Seifert surface for $B$ also called a {\it page} of the open book.

An open book can also be described by a pair $({\Sigma}, \psi)$, where ${\Sigma}$ is a compact connected oriented surface with boundary and $\psi$ an orientation-preserving diffeomorphism of ${\Sigma}$ fixing a neighborhood of $\partial {\Sigma}$, called the {\it monodromy}, by the following construction. Consider the mapping torus 
\begin{equation}\label{maptoruseq}
M(\psi) = [0,1]\times {\Sigma} / (1, x)\sim (0, \psi(x)),
\end{equation}
where, for future reference, we will write $[s, x]_\psi$ for the equivalence class of the point $(s,x)\in [0,1]\times {\Sigma}$ in $M(\psi)$. In particular, the boundary of $M(\psi)$ is $([0,1]/(0\sim 1)) \times \partial {\Sigma} = S^1\times\partial {\Sigma}$. Here and throughout, we will think of $S^1$ as $\arr / \zee = [0,1]/ 0\sim 1$ and typically use $\phi$ for the corresponding coordinate.

Attach to $M(\psi)$ the space $D^2\times \partial {\Sigma}$, a union of solid tori, whose boundary $\partial D^2 \times \partial {\Sigma}$ is identified with $\partial M(\psi)$ in the obvious way and so that $\partial {\Sigma}$ becomes the binding of an open book decomposition on the union. An open book decomposition for $Y$ is equivalent to such a pair $({\Sigma},\psi)$ along with a diffeomorphism between $Y$ and $M(\psi)\cup_\partial (D^2\times \partial {\Sigma})$, which we will generally suppress.

An open book $({\Sigma},\psi)$ is {\it compatible} with a contact structure $\xi = \ker \lambda$, and the contact structure is {\it supported} by the open book, if $d\lambda$ is a positive area form on the pages and $\lambda$ is positive on oriented tangents to $\partial {\Sigma}$. Any open book for $Y$ supports a unique isotopy class of contact structure; conversely any contact structure is isotopic to one supported by some open book. Note that the 3-manifold $Y$, and the supported contact structure, depend only on the equivalence class of $\psi$ under isotopies fixing a neighborhood of $\partial {\Sigma}$. We will write $\Mod({\Sigma})$ for the group of such isotopy classes, with the caveat that our $\Mod({\Sigma})$ consists of orientation-preserving diffeomorphisms fixing a neighborhood of $\partial {\Sigma}$, up to isotopy of such maps, even though the notation does not indicate these restrictions.

Let us fix an open book decomposition $({\Sigma},\psi)$ on $Y$ and compatible contact structure $\xi = \ker \lambda$. A knot $K\subset Y$ is {\it transverse} if $TK\pitchfork \xi$, and $\lambda$ is positive on oriented tangents to $K$. In particular, if $K$ is a component of the binding of an open book decomposition for $Y$, then $K$ is a transverse knot in the supported contact structure. Take $K$ to be a component of the binding, and for an integer $n>0$  construct a new open book $({\Sigma}', \psi'_n)$ as follows. 
\begin{enumerate}
\item Choose an annular collar neighborhood $A \cong S^1\times [0,1]$ of the boundary component of ${\Sigma}$ corresponding to $K$, and arrange by an isotopy  that $\psi$ is the identity map on $A$. We identify $S^1\times\{0\}$ with the boundary component $K$.
\item Let ${\Sigma}' = {\Sigma} \cup H$, where $H\cong [0,1]\times [0,1]$ is a 1-handle attached to two points of $S^1\times\{0\}$, so that ${\Sigma}'$ is an oriented surface with one more boundary component than ${\Sigma}$.
\item Extend $\psi$ to ${\Sigma}'$ by declaring it to be the identity on ${\Sigma}' - {\Sigma}$. Let $\psi'_n$ be the diffeomorphism given by the composition of $\psi$ with $t_{K_1}^{-1}\circ t_{\partial_1}\circ t_{\partial_2}^{n-1}$. Here $t_C$ means the right-handed Dehn twist around a simple closed curve $C$. We write $\partial_1$, $\partial_2$ for two simple closed curves in ${\Sigma}'$ parallel to the two components of $\partial {\Sigma}'$ adjacent to $H$, and $K_1$ for the curve $S^1\times 1\subset {\Sigma}$ parallel to $K$.
\end{enumerate}
Thus $({\Sigma}', \psi'_n)$ is constructed by adding a 1-handle to the boundary component corresponding to $K$, and composing $\psi$ with a left twist around $K$, a right twist along one new boundary component, and $n-1$ right twists along the other new boundary component. The relevance of this construction to our situation is the following.

\begin{lemma}\label{legsurglemma} Let $(Y,\xi)$ be a contact manifold and $\K$ an oriented Legendrian knot in $Y$. 
\begin{enumerate}
\item There exists an open book decomposition $({\Sigma}_0, \psi_0)$ supporting $\xi$ such that $\K$ lies on a page of the open book.
\item Suitable stabilization of $({\Sigma}_0,\psi_0)$ yields an open book decomposition $({\Sigma}, \psi)$ supporting $\xi$, such that a negative stabilization $\K^-$ of $\K$ lies on a page and is parallel to a binding component, which is isotopic to a transverse pushoff of $\K^-$.
\item For any $n>0$, the open book $({\Sigma}', \psi'_n)$ constructed from $({\Sigma},\psi)$ as above describes a  3-manifold with contact structure contactomorphic to the result $\xi_n^-(\K^-)$ of contact $n$-surgery along $\K^-$.
\item For any $n>0$, there is a contactomorphism between $\xi^-_n(\K)$ and $\xi^-_{n+1}(\K^-)$. 
\item The contact structure described by $({\Sigma}',\psi_n')$ is contactomorphic to the result of an inadmissible transverse surgery along the binding component $K$, and any integral inadmissible transverse surgery on $K$ can be described in this way for some $n>0$.
\end{enumerate}

\end{lemma}

\begin{proof} (1) is fairly well-known; see \cite{Etnyre:OBlectures} for example. (2) follows from Lemma 6.5 of \cite{BEVHM}. (3) is essentially the Ding-Geiges-Stipsicz algorithm \cite{DGS:surgery} describing contact surgery; this point along with (4) and (5) are spelled out by Conway \cite{Conway:transverse} (see also \cite[Lemma 2.6]{LS:surgery}).
\end{proof}

In particular, when considering properties of contact structures arising as {\it some} positive contact surgery along a Legendrian $\K$, then at the expense of possibly replacing $\K$ by $\K^-$ and increasing the contact surgery coefficient by one we may assume that $\K$ is parallel to a binding component of some open book. (Note that the transverse pushoffs of $\K$ and $\K^-$ are transversely isotopic.)

\begin{remark} In the description we have given above, the contact structure $\xi_n^-(\K)$ is described by the abstract open book decomposition $({\Sigma}',\psi_n')$. In particular, it is determined only up to contactomorphism, rather than up to isotopy. A more precise description of contact surgery can be found in \cite{Honda:classification1} and \cite{DGS:surgery}, where it is shown how to extend the contact structure $\xi$ restricted to the complement of a neighborhood of $\K$ across the torus glued in during surgery, and in particular the choices involved in doing so. In the case of contact surgery with integer coefficient $n$ having $|n|>1$, there are exactly two possibilities, corresponding to a single choice of sign. The two choices in fact determine contactomorphic (but not generally isotopic) contact structures $\xi^+_n(\K)$ and $\xi^-_n(\K)$, essentially corresponding to the choice of which boundary component to label as $\partial_2$ in the description above.
\end{remark}

With the preceding in hand, we can prove one direction of Theorem \ref{mainthm}. 

\begin{theorem}\label{capthm} Suppose $\K$ is a Legendrian knot in a contact manifold $(Y,\xi)$ with the property that for some $n>0$, the contact structure $\xi^-_n(\K)$ is weakly fillable. Then $(Y,\xi)$ is weakly fillable, and there exists a weak filling $(Z,\omega)$ of $(Y,\xi)$ containing a properly embedded symplectic disk $\Delta$ whose boundary is a positive transverse pushoff of $\K$.
\end{theorem}

\begin{proof}
Suppose $(Z', \omega')$ is a weak symplectic filling of $(Y', \xi_n^-(\K))$ (where $Y'$ is the smooth manifold underlying the result of contact surgery). By the remarks above, we may suppose that $\K$ is parallel to a binding component $K$ of an open book $({\Sigma},\psi)$ for $(Y,\xi)$, so that $\xi^-_n(\K)$ is supported by the open book $({\Sigma}', \psi'_n)$ as above, and $K$ is a transverse pushoff of $\K$. Let $K'$ be the component of $\partial {\Sigma}'$ corresponding to the component labeled $\partial_2$ in the description of $({\Sigma}', \psi'_n)$ above.

Thinking of $Y'$ as $\partial Z'$, let $Z$ be the result of attaching a 2-handle to $Z'$ along $K'$, with framing equal to that induced by the page ${\Sigma}'$. Then $\partial Z$ is the result of what is known as a ``capping off'' operation; it carries a natural open book whose page is given by the union of ${\Sigma}'$ with a disk glued along $K'$, and monodromy extended by the identity across the disk. In the resulting surface, the twists $t^{n-1}_{\partial_2}$ are isotopically trivial, while the curves $\partial_1$ and $K$ are now parallel so that the right and left twists on these curves that appear in $\psi'_n$ cancel in the capped-off open book. In other words, the capped off open book is equivalent to the original open book $({\Sigma}, \psi)$ on $Y$, and in particular $\partial Z = Y$. 

It is a consequence of a theorem of Wendl \cite[Theorem 5]{Wendl:nonexact} that the manifold $Z$ carries a symplectic form $\omega$ that weakly fills the contact structure supported by the capped-off open book $({\Sigma},\psi)$; in fact $\omega|_{Z'}$ can be taken to agree with $\omega'$ away from a small neighborhood of $Y'$. In particular, we find that $(Y, \xi)$ is weakly fillable. Moreover, Wendl's construction shows that the cocore of the 2-handle used in the construction of $Z$ is a symplectic disk $\Delta$ with (positively) transverse boundary. Topologically, this cobordism from $Y'$ to $Y$ is nothing but the trace of the surgery from $Y$ to $Y'$, turned around; from this point of view it is clear that the boundary of the cocore is smoothly isotopic to $K$. In terms of the open book, the boundary of the cocore can be seen as the braid traced out by the center of the disk used to cap off ${\Sigma}'$, and it follows from Lemma \ref{1braidlem} below that this braid is transversely isotopic to $K$ as well.  
\end{proof}

Note that the argument proving Theorem \ref{capthm} actually shows that if $(Z',\omega')$ is a weak filling of $(Y',\xi^-_n(\K))$, then by attaching a 2-handle as in the proof we find $Z'$ is symplectomorphic to the complement of a neighborhood of the (symplectic) cocore disk inside the filling $Z$ of $(Y,\xi)$. This proves the ``moreover'' statement in Theorem \ref{mainthm}.

The proof of the converse direction of Theorem \ref{mainthm} amounts to showing that one can ``undo'' the previous construction, by removing a suitable neighborhood of a symplectic disk in a weak filling. We carry this out in the following sections.

\subsection{Braids and monodromy}\label{braidsec} In our argument below it will be convenient to arrange $K$ in a particular way with respect to an open book decomposition for $(Y, \xi)$. Recall that a knot $K\subset Y$ is {\it braided} with respect to an open book if $K$ is everywhere transverse to the pages of the open book (in particular $K$ is disjoint from the binding), and intersects the pages positively. It is a theorem of Mitsumatsu and Mori \cite[Appendix]{mitsumatsu06} (see also Pavelescu \cite{pavelescu11}) that any transverse knot in $(Y,\xi)$ is isotopic, through transverse knots, to one that is braided with respect to a given open book supporting $\xi$. 

A braided transverse knot can be described in terms of the monodromy $\psi$, in a way analogous to classical braids. Namely, if ${\Sigma}$ is a chosen page of the open book and $K \cap {\Sigma}$ consists of the points $p_1,\ldots, p_k\in Int({\Sigma})$, then $K$ determines and is determined by a choice of lift of $\psi\in \Mod({\Sigma})$ to an element $\tilde\psi\in\Mod({\Sigma}, \{p_1,\ldots, p_k\})$ of the group of isotopy classes of diffeomorphism preserving the set $\{p_1,\ldots, p_k\}$, i.e., the {\it surface braid group}. In the case that $K$ is a binding component of the open book, the following shows in particular that $K$ is transversely isotopic to a 1-braid, meaning a braid that intersects each page just once, which can be described by a particular choice of lift of $\psi$.

\begin{lemma}\label{1braidlem}  Let $({\Sigma}, \psi)$ be an open book decomposition for a contact manifold $(Y, \xi)$, and $K$ the positive transverse knot given by a boundary component of ${\Sigma}$. Consider an annular collar neighborhood $A$ of that boundary component on which $\psi$ is the identity, and let $p\in {\Sigma}$ be a point in the interior of $A$. Choose disjoint boundary-parallel simple closed curves in $A$ such that $p$ is between the two curves; let $\gamma_1$ be the curve closer to the boundary and $\gamma_2$ the other. Let $\tilde\psi\in \Mod({\Sigma},p)$ be the mapping class that agrees with $\psi$ away from $A$, and is given by the composition of a right Dehn twist on $\gamma_1$ and a left Dehn twist on $\gamma_2$. Then the 1-braid represented by the pointed surface $({\Sigma}, p)$ with monodromy $\tilde\psi\in \Mod({\Sigma},p)$ is transversely isotopic to $K$. 
\end{lemma}

The situation of the Lemma is illustrated in Figure \ref{bindbraidfig}.

\begin{figure}
\begin{center}
\def\svgwidth{2.5in}
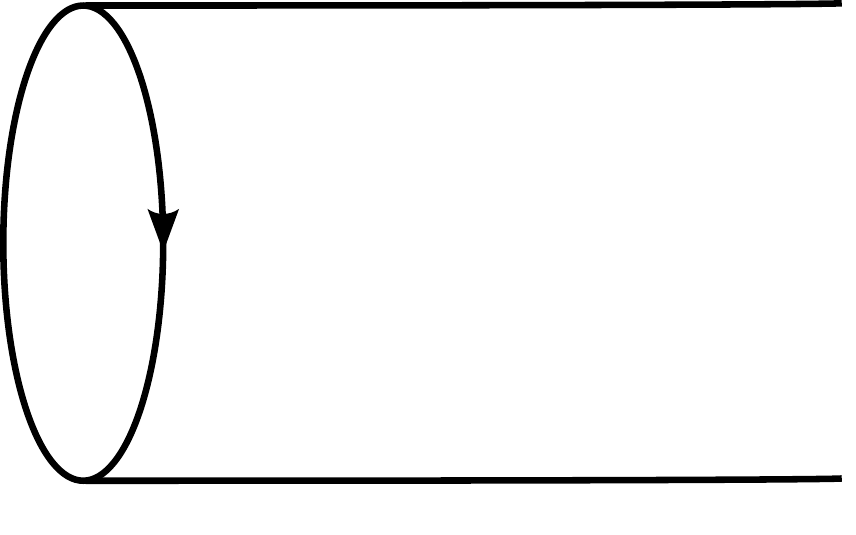
 \caption{The braid described by $p$ under the monodromy given by a right Dehn twist along $\gamma_1$ and a left twist on $\gamma_2$ is transversely isotopic to $K$.}
  \label{bindbraidfig}
\end{center}
\end{figure}

For the proof of the Lemma, as well as for certain calculations to follow, the following digression will be useful.

\subsubsection*{Digression: construction of compatible contact forms} We briefly review the construction of a contact form compatible with an open book $({\Sigma},\psi)$ (we  follow the procedure of \cite[Section 3.2]{Wendl:nonexact}). First, one finds a 1-form $\alpha$ on ${\Sigma}$ such that 
\begin{itemize}
\item $d\alpha$ is a positive area form on ${\Sigma}$.
\item $\alpha$ has a standard form near $\partial {\Sigma}$, as follows. For each boundary component of ${\Sigma}$, choose an annular neighborhood $A\cong S^1\times [\tau_0,1]$ on which $\psi$ restricts to the identity, with oriented coordinates $(\theta, \tau)$, and require $\alpha = (2-\tau)\,d\theta$. Here $\tau_0$ is a constant with $0<\tau_0<1$ and  we suppose that the boundary of ${\Sigma}$ corresponds to $\tau = \tau_0$.
\item Additionally, if $D\subset Int({\Sigma})$ is a disk with $\psi|_D = id$, with polar coordinates $(\rho, \theta)$, then we can choose $\alpha$ so that $\alpha = {1\over 2} \rho^2\, d\theta$ on $D$. 
\end{itemize}

Next, by an interpolation argument one constructs a 1-form $\alpha_\psi$ on the mapping torus $M(\psi)$ such that $d\alpha_\psi>0$ on each slice $\{\phi\}\times {\Sigma}$. Further, we can suppose that $\alpha_\psi$ agrees with (the pullback of) $\alpha$ on the subsets $S^1\times A$ and $S^1\times D$ for the annuli $A$ and disk $D$ mentioned above. Then for any small $\epsilon>0$, the form $d\phi + \epsilon \alpha_\psi$ is a contact form on $M(\psi)$ and has the form $d\phi +\epsilon(2-\tau) \, d\theta$ near the boundary tori. (Recall that $\phi$ is our usual notation for the coordinate on $S^1$ corresponding to the fibration of the mapping torus.)

Finally, one fills in the binding $D^2\times \partial {\Sigma}$ by using the coordinate identifications between $ (D^2-0)\times S^1$ and $S^1\times A$ (where a component of $\partial {\Sigma}$ has been identified with $S^1$) given by $(\rho, \phi, \theta)\mapsto (\phi, \theta, \tau = \rho)$ for $\rho\geq \tau_0$. To extend the contact form across the binding, begin by choosing a constant $\tau_0'$ with $0<\tau_0'< \tau_0$, and a 1-form $\lambda_0 = f(\rho)\, d\theta + g(\rho)\, d\phi$ on $D^2\times S^1$ , such that:
\begin{itemize}
\item $(f,g)$ defines a path in the first quadrant from $(1,0)$ to $(0,1)$ as $\rho$ increases from 0 to 1.
\item $\lambda_0$ is a smooth contact form on $\{\rho < \tau_0\} \subset D^2\times S^1$.
\item For $\rho\in [\tau_0, 1)$ we have $f(\rho) = 0$.
\item For $\rho \in [\tau_0', 1)$ we have $g(\rho) = 1$
\end{itemize}
If $\epsilon$ is chosen as above, we modify $\lambda_0$ as follows. Choose another function $f_\epsilon(\rho)$ such that 
\begin{itemize}
\item $f_\epsilon = f$ on $[0,\tau_0']$,
\item $f_\epsilon' <0$ on $[\tau_0', \tau_0]$
\item $f_\epsilon(\rho) = \epsilon(2-\rho)$ on $[\tau_0, 1]$. 
\end{itemize}
Then the 1-form $\lambda_\epsilon$ on $D^2\times S^1$ given by
\[
\lambda_\epsilon = \left\{\begin{array}{ll} f_\epsilon(\rho) \,d\theta + d\phi & \mbox{for $\rho \in [\tau_0', 1]$} \\ \lambda_0 & \mbox{for $\rho \leq \tau_0'$}\end{array}\right.
\]
is a smooth contact form on $D^2\times S^1$ that agrees with the form $d\phi + \epsilon\alpha_\psi$ when $\rho \geq \tau_0$, and with $\lambda_0$ for $\rho< \tau_0'$. In particular, note that by choosing $f(\rho) = 1$ and $g(\rho) = {1\over 2}\rho^2$ for $\rho< \tau_0''< \tau_0'$, we can obtain a contact form that in coordinates $(\rho, \phi, \theta)$ on a neighborhood $D^2\times S^1$ of a binding component has the form
\[
\lambda = d\theta + \ts{1\over 2}\rho^2\, d\phi \quad (\rho < \tau_0'').
\]

This ends the digression.

\begin{proof}[Proof of Lemma \ref{1braidlem}] It will suffice to restrict to a neighborhood of $K$. In particular we consider $nbd(K)$ to be equipped with an open book whose page is the annulus $A$ and whose monodromy is the restriction of $\psi$, i.e., the identity. Let $\psi_s$, $s\in [0,1]$ be an isotopy supported in $A$ with $\psi_0 = \psi$ and $\psi_1 = \tilde\psi$ (the isotopy is relative to a small neighborhood of $\partial A$, but of course moves the point $p$). An isomorphism between the mapping tori of $\psi = \id$ and $\tilde\psi$ (as maps of $A$) is then given by
\begin{align*}
M(\tilde\psi) &\to M(\id)\\
[s, a]_{\tilde\psi} & \mapsto [s, \psi_s(a)]_{\id},
\end{align*}
using the notation introduced after \eqref{maptoruseq}. In particular the braid in $M(\psi)$ parametrized by $s\mapsto [s, p]_{\tilde\psi}$ corresponds to the curve $s\mapsto [s,\psi_s(p)]_{\id}$ in $M(\id)$. 

Choose oriented coordinates $(\theta, \tau)$ on $A = S^1\times [0,1]$, where $\tau = 0$ corresponds to the boundary component $K$. This means that $\tau$ is an {\it inward}-directed coordinate on the collar, and the orientation on $K$ as the boundary of ${\Sigma}$ is given by the direction of {\it increasing} $\theta$. By the digression above, we can suppose that the contact structure on $nbd(K)$ is given by identifying $nbd(K) = S^1\times D^2$ with coordinates $\theta$ on $S^1$ and $(\rho, \phi)$ polar coordinates on $D^2$,  and for a chosen $\tau_0''<1$, the contact form is given by $\lambda = d\theta + {1 \over 2}\rho^2 \, d\phi$ on the set $\{\rho\leq \tau_0''\}$. The open book projection $(\theta, \rho, \phi)\mapsto \phi$ supports this contact structure, and the page is the annulus with coordinates $(\theta, \rho)$. Thus the mapping torus $M(\id)$ can be identified (after collapsing the boundary component $\{\tau = 0\}$ to a circle) with $S^1\times D^2$ by the natural map $[s, (\theta, \tau)]_{\id}\mapsto (\theta, \rho=\tau, \phi = s)$. 

For convenience, let us assume $p\in A$ is given by the point $(\theta, \tau) = (0, T)$ for some $T < \tau_0''$. By writing an explicit model for the twists on $\gamma_1$ and $\gamma_2$, it is not hard to see that the isotopy $\psi_s$ can be chosen so that in the given coordinates the path $s\mapsto \psi_s(p)$ is $s\mapsto (\theta = s, \tau = T)$. Hence, by the identification between $M(\tilde\psi)$ and $M(id)$ above,  the braid under consideration in $M(\id)$ is the oriented curve $c(s)= (\theta = s, \tau = T, \phi = s)$. The tangent to this curve clearly has $\lambda(c'(s)) >0$; in fact if we take $c_\tau(s) = (\theta = s, \tau, \phi=s)$ then letting $\tau$ decrease from $T$ to $0$ gives a transverse isotopy from the 1-braid to the oriented binding $\{\tau = 0\}$. 
\end{proof}

\subsection{Carving symplectic disks} We continue with the notation of the previous section: we are given a contact 3-manifold $(Y,\xi)$ with a weak symplectic filling $(Z,\omega)$ containing a properly embedded symplectic disk $\Delta$ whose boundary is a positively transverse knot $K\subset Y$. The goal of this section and the next is to prove:

\begin{theorem}\label{scoopthm} The symplectic disk $\Delta$ has an arbitrarily small open neighborhood $U$, such that if $Z' = Z - U$ and $\omega' = \omega|_{Z'}$, then $(Z', \omega')$ is a weak symplectic filling of a contact structure $\xi'$ on $Y' = \partial Z'$. Moreover, the contact structure $\xi'$ is obtained from $\xi$ by an inadmissible transverse surgery along $K$, or (equivalently) from a positive contact surgery along a Legendrian approximation of $K$.
 \end{theorem}

This is essentially a restatement of Theorem \ref{scoopthmintro}. Taken together, this theorem and Theorem \ref{capthm} imply Theorem \ref{mainthm}. 

As preparation, we can assume by an isotopy that the transverse knot $K$ is a 1-braid with respect to an open book decomposition of $(Y,\xi)$. In particular, we can suppose that $K$ is determined by a lift of the monodromy $\psi$ to an element of the mapping class group of ${\Sigma}$ relative to a point $p$ (as well as $\partial {\Sigma}$), which we still call $\psi$. In fact, we can assume that this lift fixes a disk neighborhood $D\subset {\Sigma}$ of $p$, and call the resulting choice $\hat\psi\in \Mod({\Sigma}, D)$. The choice of $\hat\psi$ is essentially equivalent to a choice of framing on $K$; any desired framing can be realized by composing $\hat\psi$ with Dehn twists around $\partial D$. The complement of the binding in $Y$ is identified with the mapping torus $M(\hat\psi) = [0,1]\times {\Sigma}/\sim$, and the choice of lift $\hat\psi$ determines an identification $nbd(K) \cong S^1\times D$. 

\begin{itemize}
\item Coordinates on $nbd(K)$ will be $(\kappa_1, \rho, \kappa_2)$, where $(\rho, \kappa_2)$ are polar coordinates on $D$ with $\rho \leq 1$. All angular coordinates are taken in $\arr / \zee$. 
\item The open book projection $\pi: Y-B \to S^1$ is given in $nbd(K)$ by $\pi(\kappa_1, \rho, \kappa_2) = \kappa_1$, so the pages of the open book are tangent to $\ker (d\kappa_1)$.
\item By the digression in section \ref{braidsec},  in this neighborhood of $K$ the contact form can be taken to be 
\[
\lambda_Y = d\kappa_1 + \epsilon \rho^2 \, d\kappa_2
\]
for any small $\epsilon$ (since in $nbd(K)$ we have $\phi = \kappa_1$).
\end{itemize}

Now turn to the weak filling $Z$. The disk $\Delta$ has a neighborhood diffeomorphic to $\Delta_1\times \Delta_2$, where $\Delta_i$ is a unit disk with coordinates $(r_i, \theta_i)$ for $ i = 1,2$. We can suppose that under this diffeomorphism $\partial \Delta_1\times\Delta_2$ is identified with $nbd(K) = S^1\times D$, though the framings need not agree. In particular, after an isotopy we can assume that this identification is via a diffeomorphism 
\begin{align*}
F: \partial \Delta_1\times \Delta_2 &\to nbd(K) = S^1 \times D\\
F: (r_1=1, \theta_1, r_2, \theta_2) & \mapsto (\kappa_1 = \theta_1, \rho = r_2, \kappa_2 = \theta_2 + n\theta_1)
\end{align*}
for some integer $n$ representing the difference between the framing on $K$ induced by the disk $\Delta_1$ and that corresponding to the choice of $\hat\psi$. For later use, we will always assume that $\hat\psi$ has been chosen so that $n>0$.

Recall that $\Delta\subset Z$ is assumed to be symplectic; by rescaling the form $\omega$ we may suppose it has symplectic area $\pi$. By standard neighborhood theorems, since the normal bunde of $\Delta$ is (symplectically) trivial there is a symplectomorphism between some tubular neighborhood $U_Z$ of $\Delta$ in $Z$ and a neighborhood $U_\Delta$ of $\Delta_1\times 0\subset \Delta_1\times\Delta_2$, where the symplectic structure on the latter is given by the standard split form
\[
\omega_0 = r_1\,dr_1\wedge d\theta_1 + r_2\, dr_2\wedge d\theta_2.
\]
Choose $r_0$ small enough that the subset $\{r_2\leq r_0\}$ lies in $U_\Delta$, so that we can symplectically identify that set with a subset of $Z$. In the following we will work with the standard structure above on $\Delta_1\times \Delta_2$ with the understanding that our construction should be restricted to the subset $U_\Delta$ and transported via symplectomorphism into $U_Z\subset Z$.

Now, $(Z,\omega)$ is assumed to be a weak filling of $(Y,\xi)$, where $\xi$ is a contact structure supported by the open book determined by $({\Sigma},\psi)$. Observe that varying the choice of $\epsilon$ in the construction of such a contact form, as in the digression in Section \ref{braidsec}, corresponds to an isotopy of $\xi$; by Gray stability such an isotopy can be realized by an isotopy of $Y$. A diffeomorphism of $Z$ that is the identity away from a collar neighborhood of $Y$ and effects this isotopy in the collar will pull back $\omega$ to a symplectic form that weakly fills the modified contact structure, and determine a symplectic structure on $Z$ deformation equivalent to the original. In particular the constant $r_0$ does not change through this process, and the conclusion is that we may choose $\epsilon$ as small as desired without affecting $r_0$.

 Consider a solid torus $D^2\times S^1$ with coordinates $(\sigma, \alpha, \beta)$, where $(\sigma, \alpha)$ are polar coordinates on the unit disk $D^2$ with $\alpha, \beta\in \arr / \zee$. We embed $D^2\times S^1$ in $\Delta_1\times \Delta_2$ by a map 
\begin{align}
\nonumber i: D^2\times S^1 &\to \Delta_1\times \Delta_2\\
\label{incldef} i: (\sigma, \alpha, \beta) &\mapsto (r_1 = r_1(\sigma), \theta_1 = \alpha, r_2 = r_2(\sigma), \theta_2 = \beta),
\end{align}
where $r_1(\sigma)$ and $r_2(\sigma)$ are smooth functions described below. We denote the image of this embedding by $H \subset \Delta_1\times \Delta_2$, and usually identify $H = D^2\times S^1$ via $i$. Schematically, $H$ is described by the diagram below.
\[
\def\svgwidth{3in}
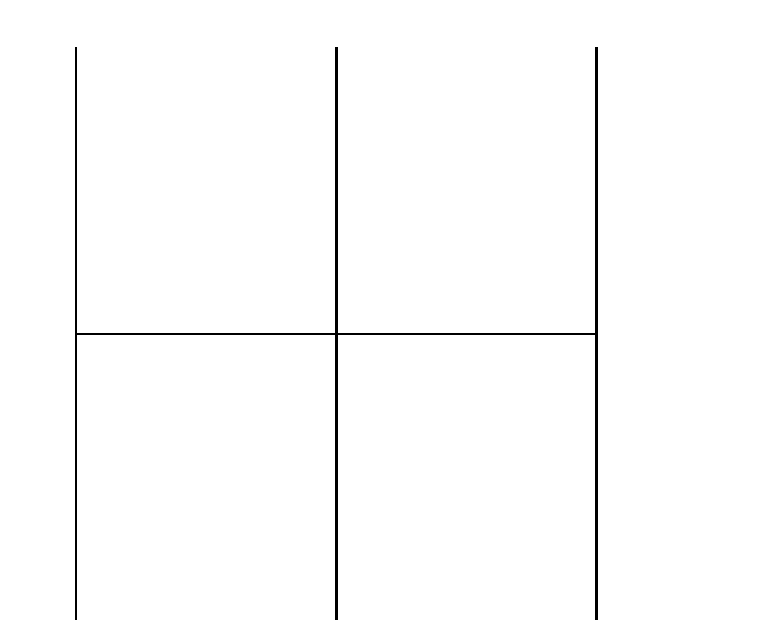
\]
The functions $r_1(\sigma)$, $r_2(\sigma)$ are selected as follows. We suppose the constant $r_0\in (0,1)$ is given as previously.  Fix constants $0< \sigma_0<\sigma_1< r_0< 1$, and assume $\epsilon$ is chosen small enough to satisfy:
\begin{equation}\label{epsilonreq}
\epsilon < \min\{\sigma_0, \sigma_1 - \sigma_0, 1/2n\sigma_1^2\}.
\end{equation}
Now require that $r_1(\sigma)$ is a smooth function satisfying
\[
r_1(\sigma) = \left\{\begin{array}{ll} k\sigma & 0\leq \sigma\leq \sigma_0 \\ 1 & \sigma_1-\epsilon\leq \sigma\leq 1\end{array}\right.
\] 
and otherwise increasing, where $k>0$ is a constant with $k\sigma_0 < 1$. Choosing an additional constant $r_0'$ with  $0< r_0'< r_0 < 1$, we select $r_2(\sigma)$ to be a smooth, strictly increasing function such that
\begin{align*}
r_2(0) &= r_0'\\
r_2(\sigma) & \leq r_0 \quad 0\leq \sigma\leq \sigma_1\\
r_2(\sigma) &= \sigma \quad \sigma_1\leq \sigma \leq 1
\end{align*}

The hypersurface $H$ separates $\Delta_1\times \Delta_2$ into two subsets $U_1$ and $U_2$, where $U_1$ is a neighborhood of $\Delta_1\times 0$ and $U_2$ is the complementary region. Identifying $\Delta_1\times\Delta_2$ with a subset of $Z$ as above, we can think of $U_1\subset Z$, and define $Z' = Z - U_1$. Of course $Z'$ is independent of the choices of the constants and embeddings in the above, up to diffeomorphism. 

The boundary $Y' = \partial Z'$ is obtained from $Y$ by removing $nbd(K)$ and replacing it by $H$. The coordinates on $nbd(K)$ and $H$ are related via the embedding $i$ and map $F$ above. Explicitly, on the subset $\{\sigma\geq \sigma_1\}\subset H$ we have the transformation
\begin{equation}\label{Htransf}
(\sigma, \alpha, \beta) \mapsto (\kappa_1 = \alpha, \rho = r_2(\sigma), \kappa_2 = \beta + n\alpha),
\end{equation}
where in fact, in this range we have $r_2(\sigma) = \sigma$.

Observe: The coordinates $(\kappa_1, \rho, \kappa_2)$ on $nbd(K)$ are positively oriented for $Y$. The coordinates $(\sigma, \alpha, \beta)$ on $H$ correspond to the orientation on $H$ induced as the boundary of $U_1$. On the region $\sigma\geq \sigma_1$ these are opposite orientations, which is reflected in the fact that the transformation above is orientation-reversing. 

Finally, recall that the open book on $Y$, restricted to $nbd(K)$, is given by the map $\pi(\kappa_1, \rho, \kappa_2) = \kappa_1$. By the above, this corresponds to the map $\pi': (\sigma, \alpha, \beta)\mapsto \alpha$, which extends to the complement of $0\times S^1$ in $H = D^2\times S^1$ as an open book with a binding component $B' = 0\times S^1$. Together with the remainder of the original open book on $Y$ away from $nbd(K)$ this defines an open book $({\Sigma}', \psi')$ on $Y'$, and we let $\xi'$ be the (unique isotopy class of) contact structure on $Y'$ supported by this open book. Observe that the page of this open book is the surface ${\Sigma}' = {\Sigma} - D$. We will see that the monodromy $\psi'$ is obtained from the restriction of $\hat\psi$ to ${\Sigma}'$ by composing with $n$ right-handed Dehn twists around a curve parallel to $\partial D$. Before doing so, we verify that the restriction of $\omega$ to $Z'$ indeed gives a weak filling of the contact structure supported by $({\Sigma}', \psi')$.

Consider the following Liouville vector field on $\Delta_1\times \Delta_2$, defined away from $\Delta_1\times 0$:
\[
X = {1\over 2}r_1\, \partial_{r_1} + ({1\over 2}r_2 - {c \over {r_2}})\,\partial_{r_2},
\]
where $c>0$ is a constant. Note that $X$ is directed radially outward in the direction of $\Delta_1$, and radially inward along $\Delta_2$ so long as $r_2< \sqrt{2c}$. In particular, so long as $c> {1\over 2}r_0^2$, we have that $X$ is transverse to $H$ and directed into $U_1$ (i.e., directed out of $Z'$), and we fix $c$ satisfying this condition.

The primitive for $\omega_0$ associated to $X$ is
\[
\iota_X\omega_0 = {1\over 2} r_1^2\,d\theta_1 + ({1\over 2} r_2^2 - c) \, d\theta_2,
\]
which pulls back to $H$ as the form
\[
\lambda_0= i^*(\iota_X\omega_0) = {1\over 2} r_1(\sigma)^2 \, d\alpha + ({1\over 2}r_2(\sigma)^2 - c) \, d\beta.
\]
Now, recall that $(\sigma, \alpha, \beta)$ are coordinates that orient $H$ as the boundary of $U_1$, not the boundary of $Z'$. Using the positive coordinate system $(\sigma, \alpha, \beta')$, where $\beta' = -\beta$, we see that:
\begin{itemize}
\item The form $d\lambda_0 = r_1r_1' \, d\sigma\wedge d\alpha - r_2r_2'\,d\sigma\wedge d\beta'$, restricted to a page of the open book $\pi'(\sigma,\alpha, \beta') = \alpha$, is nondegenerate. Indeed, the  tangent to the page is given by $\ker(d\alpha)$, co-oriented by $d\alpha$, and we have 
\[
d\lambda_0 \wedge d\alpha = r_2r_2' d\sigma\wedge d\alpha\wedge d\beta'.
\]
Since $r_2(\sigma)$ is strictly positive and increasing, this is positive. 
\item In particular a positive ordered basis for the tangent space to a page is given by $\{\partial_{\beta'}, \partial_\sigma\}$ (since this basis followed by $\partial_\alpha$ is positive for $Y'$). Since $\partial_\sigma$ is an inward-pointing vector along the boundary of the page, it follows that $\partial_{\beta'}$ orients the boundary of the page. 
\item We have $\lambda_0(\partial_{\beta'}) = c - {1\over 2}r_2(\sigma)^2$. Since $r_2(0) = r_0' < r_0$ and $c > {1\over 2}r_0^2$, this is positive along the binding $B' = \{\sigma = 0\}$. 
\end{itemize}

Thus, $\lambda_0$ defines a contact structure on $H$ that is compatible with the open book we have introduced on $H$. It remains to ``patch'' the contact structure $\ker(\lambda_0)$ together with the given contact structure on $Y$; recall that the latter is given by the kernel of the form $\lambda_Y = d\kappa_1 + \epsilon\rho^2\, d\kappa_2$ on $nbd(K)$, a subset of which is identified with the region $\{\sigma\geq \sigma_1\}\subset H$. Concretely, the identification is given by a map $j: \{\sigma\geq \sigma_1\}\to \{\rho \geq \sigma_1\}$ where
\[
j: (\sigma, \alpha, \beta) \mapsto (\kappa_1 = \alpha, \rho = \sigma, \kappa_2 = -\beta' + n\alpha).
\]
Thus we get a form $\lambda_{Y'} = j^*\lambda_Y$ given by
\begin{align}
\lambda_{Y'} &= d\alpha + \epsilon\sigma^*(-d\beta' + nd \alpha)\nonumber\\
&= -\epsilon \sigma^2\,d\beta' + (1+\epsilon n\sigma^2)\,d\alpha \qquad (\sigma\geq \sigma_1).\label{region1}
\end{align}
On the other hand, we have seen that on $H$,
\[
\lambda_0 = (c-{1\over 2}r_2(\sigma)^2)\,d\beta' + {1\over 2}r_1(\sigma)^2\,d\alpha,
\]
which for $\sigma\leq \sigma_0$ reduces to
\begin{equation}\label{region2}
\lambda_0 = (c-{1\over 2}r_2(\sigma)^2)\,d\beta' + {1\over 2}k^2\sigma^2\,d\alpha \qquad (\sigma\leq \sigma_0)
\end{equation}
Now, a smooth 1-form on $H$ given by $\tilde\lambda = f(\sigma)\, d\beta' + g(\sigma)\, d\alpha$ is a positive contact form if the curve $(f(\sigma), g(\sigma))$ in the plane winds counterclockwise about the origin in the sense that $fg' - gf' >0$ for all $\sigma$. Bearing in mind the desired behavior of $f$ and $g$ as in \eqref{region1} and \eqref{region2}, we first choose $f$  such that:
\begin{align*}
\begin{array}{ll}f(\sigma) =  c-{1\over 2}r_2(\sigma)^2 &(\sigma\leq \sigma_0) \\ f(\sigma ) = -\epsilon \sigma^2 & (\sigma\geq \sigma_1) \\ f'(\sigma) < 0 & (\sigma > 0) \end{array} 
\end{align*}
Indeed, recall that $r_2(\sigma)$ is chosen so that $r_0' \leq r_2(\sigma) \leq r_0 < \sqrt{2c}$ for $0\leq \sigma\leq \sigma_1$ (and $r_2(\sigma)$ is increasing), so a smooth interpolation exists as required. Moreover, we may suppose that the only zero of $f$ is at the value $\sigma = \sigma_1 - \epsilon$, and since $\epsilon < \sigma_1$ we can arrange that for $\sigma \in [\sigma_1 - \epsilon, \sigma_1]$ we have 
\begin{equation}\label{fprimecond}
|f'(\sigma)|\geq \epsilon\sigma_1.
\end{equation}


For the function $g$, since $k\sigma_0 < 1$ we can arrange
\[
\begin{array}{ll} g(\sigma) = {1\over 2}k^2\sigma^2 & (\sigma\leq \sigma_0) \\ g(\sigma) = 1+\epsilon n\sigma^2 & (\sigma\geq \sigma_1-\epsilon) \\ g'(\sigma) > 0 & (\sigma > 0) \end{array}
\]

By construction, $\tilde\lambda$ agrees with $\lambda_0$ on $\{\sigma\leq \sigma_0\}$ and with $\lambda_{Y'}$ on $\{\sigma\geq \sigma_1\}$. The winding condition holds on $\{\sigma_0\leq\sigma\leq\sigma_1-\epsilon\}$ since $fg'>0$ and $gf'<0$ in that interval. Finally, while $f(\sigma), f'(\sigma)<0$ for $\sigma \in (\sigma_1-\epsilon, \sigma_1]$, it is easy to check that the constraints \eqref{epsilonreq}, \eqref{fprimecond} imply the winding condition holds in this interval as well.


 Hence $\tilde\lambda$ is a positive contact form on $H$, and together with $\lambda_Y$ on $Y- nbd(K) = Y' - H$ gives a contact form on all of $Y'$ that agrees with the form $\lambda_Y$ away from $H$.

\begin{proposition}\label{weakfillprop} The contact structure $\xi' =\ker(\tilde\lambda)$ is compatible with the open book on $Y'$ constructed above, and is weakly filled by the symplectic structure on $Z'$ obtained by the restriction of $\omega$.
\end{proposition}

\begin{proof} Compatibility with the open book means that $\tilde\lambda$ is positive on the oriented boundary of the pages, and that $d\tilde\lambda$ is an area form on the interior of the pages. Since $\tilde\lambda = \lambda_Y$ away from $H$, and $\lambda_Y$ is compatible with the original open book, we need only check these conditions in $H$. Positivity on the new binding component $B'\subset H$ follows since $\tilde\lambda = \lambda_0$ near $B'$, and we have already checked this condition for $\lambda_0$. 

For positivity of $d\tilde\lambda$ on the pages, write
\[
d\tilde\lambda = f'(\sigma)\,d\sigma\wedge d\beta' + g'(\sigma)\, d\sigma\wedge d\alpha.
\]
Hence $d\tilde\lambda \wedge d\alpha = -f'(\sigma)\, d\sigma\wedge d\alpha\wedge d\beta' >0$ since $f'< 0$, and the claim follows since the oriented tangents to the page are $\ker(d\alpha)$.

The weak filling claim is the assertion that $\omega>0$ on the oriented contact planes $\ker(\tilde\lambda)$. This is true by assumption away from $H$, and indeed on $\{\sigma\geq r_0\}$. In $\{r_2\leq r_0\}\subset \Delta_1\times \Delta_2\subset Z$ the symplectic form is our standard one $\omega_0$, so to check positivity it suffices to show that $i^*\omega_0 \wedge \tilde\lambda$ is positive, where $i: H\to \Delta_1\times\Delta_2$ is the embedding from previously. We calculate using \eqref{incldef} (and $\beta' = -\beta$):
\begin{align*}
i^*\omega_0 \wedge \tilde\lambda & =( r_1(\sigma)r_1'(\sigma) \, d\sigma\wedge d\alpha - r_2(\sigma) r_2'(\sigma)\, d\sigma\wedge d\beta') \wedge (f(\sigma)\, d\beta' + g(\sigma)\, d\alpha)\\
&= (r_1(\sigma)r_1'(\sigma) f(\sigma) + r_2(\sigma) r_2'(\sigma)g(\sigma))\, d\sigma\wedge d\alpha\wedge d\beta'.
\end{align*}
Now, this is automatically positive when $\sigma\leq \sigma_0$, since in that region $\tilde\lambda = \lambda_0$ is induced by a Liouville field for $\omega$. More generally, recall $r_2$ and $r_2'$ are positive for $\sigma>0$, and $g(\sigma)>0$ for all $\sigma$ as well. We have $r_1(\sigma)>0$ and $r_1'(\sigma)>0$ when $\sigma<\sigma_1-\epsilon$, and on this region $f(\sigma)>0$. When $\sigma\geq \sigma_1-\epsilon$ the function $f$ is negative, but $r_1'(\sigma) = 0$ in this interval. Hence the first term in parentheses above is nonnegative while the second is positive, finishing the proof.

\end{proof}

The proposition above proves Theorem \ref{scoopthm} except for the claim that the contact manifold $(Y',\xi')$ is obtained from $(Y,\xi)$ by positive contact surgery on a Legendrian approximation of $K$. We turn to that question next.

\subsection{Monodromy and framing}

Recall that we have chosen a lift of the monodromy $\psi\in\Mod({\Sigma})$ of a given open book on $Y$ to an element $\hat\psi\in\Mod({\Sigma}, D) \cong \Mod({\Sigma}')$ (all mapping class groups are implicitly taken relative to the boundary), which entails a choice of framing of the 1-braid $K$ that we call the ``page framing'' even though, strictly, this framing is not determined by the open book or even the representation of $K$ as an element of $\Mod({\Sigma}, p)$. We have denoted the framing induced by the disk $\Delta\subset Z$, relative to the page framing, by $n$, and wish to understand the monodromy of the open book on $Y'$ determined by the construction above. To do so we must be careful with the notion of ``relative monodromy.''

Let $\pi: M\to S^1$ be a fiber bundle projection, where we think of $S^1 = [0,1]/(0\sim 1)$ with quotient map $q:[0,1]\to S^1$. If $F = \pi^{-1}([0])$ is the fiber over the base point, then the bundle $\tilde\pi: q^*M\to [0,1]$ over $[0,1]$ is isomorphic to $[0,1]\times F$ by some bundle isomorphism $\varphi: q^*M\to [0,1]\times F$ that we may assume to be the identity on $F = \{0\}\times F$. Observe that by the construction of an induced bundle, the fibers of $q^*M$ over $0$ and $1$ are canonically identified with $F = \pi^{-1}([0])$. The monodromy of $\pi$ is then defined to be the composition
\[
 F = \{1\}\times F \to\tilde\pi^{-1}(1) = \pi^{-1}([0]) =  F
\]
where the arrow is determined by $\varphi$. If $\mu$ is this composition, then $\mu$ is well-defined up to conjugation (corresponding to the choice of identification of the fiber, which we will assume to be fixed from now on) and isotopy (corresponding to the choice of $\varphi$). The bundle $M$ is then isomorphic to the mapping torus $M(\mu) = [0,1]\times F / (1, x)\sim (0, \mu(x))$.

Now suppose $B\subset F$ is a subset and $\mu: F\to F$ is a diffeomorphism that is the identity on $B$. (In our situation, $F$ will be a compact surface and $B$ a collar of its boundary.) Then the mapping torus $M(\mu)$ contains a sub-bundle with fiber $B$, namely $[0,1]\times B / (0,b)\sim(1,b)$. In particular, the map $S^1\times B\to M(\mu)$ sending $(t, b)$ to the same element considered in $M(\mu)$ is a canonical trivialization of this sub-bundle. Moreover, modifying $\mu$ by an isotopy that is fixed on $B$ to another diffeomorphism $\mu'$ gives rise to a bundle isomorphism $M(\mu)\to M(\mu')$ that respects the corresponding trivializations of the sub-bundles. In other words, an element of $\Mod(F,B)$ gives rise to a well-defined bundle and sub-bundle pair, where the sub-bundle is trivialized. Conversely, suppose $\pi: M\to S^1$ is a fiber bundle with fiber $F = \pi^{-1}([0])$, and $N\subset M$ is a sub-bundle  such that $N\cap F = B$. Assume also that $N$ is a trivial bundle, and that a trivialization $S^1\times B\to N$ has been chosen. Then, so long as $B\subset F$ is a reasonable subset (e.g. a closed submanifold), one can choose the trivialization $\varphi$ of $q^*M$ above so as to restrict to the given trivialization of $q^*N$. In particular we obtain a monodromy diffeomorphism that is the identity on $B$, well-defined up to isotopy rel $B$: that is, an element of $\Mod(F, B)$ that we call the relative monodromy.

The point of the preceding discussion is that relative monodromy is determined not just by the bundle $\pi: M\to S^1$ with its trivial sub-bundle, but by the additional choice of a trivialization of the sub-bundle. Note that the notion of monodromy for an open book decomposition is, strictly, an instance of relative monodromy. An example that is relevant for our situation is the following. Let $A = [0,1]\times S^1$ be an annulus, with oriented coordinates $(\tau, \phi)$ (we take $\phi \in [0,1]/(0\sim 1)$ as usual). For an integer $k$, let $\psi_k: A\to A$ be the diffeomorphism $\psi_k(\tau, \phi) = (\tau, \phi - k\tau)$, which is isotopic rel $\partial A$ to the $k$-th power of a right-handed Dehn twist about the core circle of $A$. (Strictly, it is more appropriate to replace $-k\tau$ in this definition by a monotonic smooth function equal to $0$ for $\tau$ near $0$ and $-k$ for $\tau$ near $1$.) The mapping torus $M(\psi_k)$ is then described concretely as
\[
M(\psi_k) = [0,1]\times [0,1]\times S^1 / (1, \tau, \phi)\sim (0, \tau, \phi-k\tau),
\]
with bundle projection $[s, \tau, \phi]_{\psi_k} \mapsto s$. Here we write $[\cdot]_{\psi_k}$ for the equivalence class modulo the relation induced by $\psi_k$ as above. As a bundle over $S^1$, the trivial mapping torus $M(\id)$ is isomorphic to $M(\psi_k)$ by the map 
\begin{align*}
G: M(\id) & \to M(\psi_k)\\
[s, \tau, \phi]_{\id} &\mapsto [s, \tau, \phi+k\tau s]_{\psi_k}.
\end{align*}
In particular, we can think of $M(\psi_k)$ as just the trivial bundle $M(\id)$, but with different trivialization on the boundary to account for the relative monodromy. Indeed, the trivialization of $\partial M(\psi_k)$ is given by the obvious map $[0,1]\times \partial A/\sim_{\psi_k}\to S^1\times\partial A$ sending $[s, j, \phi]_{\psi_k}$ to $(s, \{j\}\times \phi)$ for $j = \{0,1\}$ (thinking of $\partial A = \{0,1\}\times S^1$). Composing with $G$, the corresponding trivialization $\partial M(\id)\to S^1\times \partial A$ maps
\begin{align*}
[s, 0, \phi]_{\id} &\mapsto (s, \{0\}\times\phi)\\
[s, 1, \phi]_{\id} & \mapsto (s, \{1\}\times (\phi + ks)).
\end{align*}
Put another way, the trivial bundle $M(\id)$, equipped with these boundary trivializations, has relative monodromy isotopic to $k$ right Dehn twists. 

Returning to the setting of open book decompositions, to say that $Y$ is equipped with an open book decomposition $({\Sigma}, \psi)$ means that the complement of a small neighborhood of the binding $B$ is identified with the mapping torus ${\Sigma}_\psi$. Moreover, the bundle structure on $Y-nbd(B)$ is trivialized at the boundary using the meridians of $B$.  Likewise, in the setting from previously, the choice of lift $\hat\psi\in\Mod({\Sigma}, D)$ gives rise to a trivialization $nbd(K) \cong S^1\times D$ that we have recorded in the coordinate system $(\kappa_1, \rho, \kappa_2)$. 

The manifold $Y'$ is obtained by replacing $nbd(K)$ by $H \cong D^2\times S^1$ in such a way that near $\partial H$ the coordinates $(\sigma, \alpha, \beta)$ on $H$ are related to those on $nbd(K)$ via the transformation \eqref{Htransf}:
\[
(\sigma, \alpha, \beta') \mapsto (\kappa_1 = \alpha, \rho = \sigma, \kappa_2 = \beta + n\alpha).
\]
The open book $\pi'$ on $Y'$ is described on $H$ by the map $\pi'(\sigma, \alpha,\beta) = \alpha$. By a slight abuse, we think of this as defined for $\sigma\in [0,1]$ (not just the half-open interval), so the fiber of $\pi'$ is an annulus $A = [0,1]\times S^1$ described by coordinates $(\sigma, \beta)$. In other words, after deleting a small open neighborhood of the binding $B' = \{\sigma = 0\}$ and recoordinatizing, the new open book is described in this coordinate system as a trivial mapping torus $S^1\times A$ with coordinates $(\alpha, \sigma, \beta)$. (Note that this coordinate system reverses the orientation on $H$, which as we have seen makes its orientation consistent with that of $Y'$.) On the boundary component at $\sigma = 0$, the natural trivialization of this bundle given by projection to $\alpha$ and $\beta$ coordinates corresponds to the trivialization given by meridians of $B'$, as required for an open book. At $\sigma = 1$, the trivialization is dictated by the coordinate transformation above, and is the map $S^1\times \{0\}\times S^1\to S^1\times S^1$ given by $(\alpha, 1, \beta) \mapsto (\alpha, \beta + n\alpha)$. Comparing with the example above, we obtain:

\begin{proposition}\label{monodromyprop} The open book $({\Sigma}', \psi')$ on $Y'$ obtained by the procedure above has monodromy described as follows. Write ${\Sigma}' = ({\Sigma} - D)\cup A$ where $A$ is an annulus glued to ${\Sigma}$ along $\partial D$. The monodromy $\psi'$ is equal to the chosen lift $\hat\psi$ on ${\Sigma} - D$, and on $A$ is given by the composition of $n$ right-handed Dehn twists around the core of $A$. 
\end{proposition}

Recall that there are various choices for the lift $\hat\psi$, related to each other by Dehn twists around $\partial D$. Adding such a Dehn twist to $\hat\psi$ changes the corresponding framing on $K$, however, in particular adding a right Dehn twist increases the ``page framing'' by 1. The framing of $K$ induced by $\Delta$ is independent of this choice, but the integer $n$ represents the difference between the framing given by $\Delta$ and the page framing. Hence adding a right twist along $\partial D$ to $\hat\psi$ means that $n$ {\it decreases} by 1, and the monodromy $\psi'$ described in the proposition is well-defined.

\begin{lemma} Let  $K\subset Y$ be a positively transverse knot in a contact manifold $(Y, \xi)$, and let $\lambda$ be a chosen framing of $K$. Then there exists an open book on $Y$ supporting $\xi$, such that $K$ is transversely isotopic to a binding component, and such that the framing on $K$ induced by the page of the open book is smaller than $\lambda$.
\end{lemma}

\begin{proof} Recall that the construction of \cite{BEVHM} arranging for a transverse knot $K$ to be a binding component of an open book begins with taking a Legendrian approximation $\K$ of $K$, and finding an open book such that this approximation lies on a page in such a way that the framing induced by the page equals the contact framing of $\K$. Then, by a suitable stabilization, one obtains an open book with a binding component $B$ that is a transverse pushoff of $\K$ (and hence transversely isotopic to $K$) and furthermore a curve $\K^-$ on the page parallel to that binding component is a Legendrian approximation of $B$ (and a negative stabilization of $\K$).  Altering $\K$ by a negative stabilization does not change the transverse isotopy class of its transverse pushoff. Thus, by taking many negative stabilizations of $\K$ at the beginning of the construction, we can suppose that the contact framing of $\K^-$ is less than $\lambda$, in fact as much less as we desire. The contact framing of $\K^-$ is the same as the framing induced by the page, which then corresponds to the framing of $B \simeq K$ induced by the page.
\end{proof}

\begin{proposition}\label{contsurgprop} Let $(Z,\omega)$ be a weak symplectic filling of $(Y, \xi)$, and $\Delta\subset Z$ a properly embedded symplectic disk with (positively) transverse boundary $K\subset Y$. Let $Z' = Z - nbd(\Delta)$ be the symplectic manifold weakly filling $(Y', \xi')$ as in Proposition \ref{weakfillprop}. Then the contact structure $\xi'$ is obtained from $\xi$ by an inadmissible transverse surgery along $K$. Alternatively, $\xi'$ can be described as the result of a positive contact surgery along a Legendrian approximation of $K$.
\end{proposition}

\begin{proof} Let $\lambda$ denote the framing on $K$ induced by $\Delta$, and apply the transverse isotopy of previous lemma to arrange that $K$ is a binding component of an open book for $(Y, \xi)$, such that the page framing is less than $\lambda - 1$. Note that by attaching the trace of the isotopy to $\Delta$ in a collar attached to $\partial Z$, we can still suppose that $\partial \Delta = K$ (as in the proof of Corollary \ref{existencecor}(a) in Section \ref{constructionsec}; see \cite{etnyregolla}).

Now apply the isotopy of Lemma \ref{1braidlem} to realize $K$ as a braid adjacent to the boundary of the open book as in Figure \ref{bindbraidfig}, with the associated lift of the monodromy to a class in $\Mod({\Sigma},D)$. (Again we carry $\Delta$ through the isotopy.) We must see how the framing induced by this lift relates to the page framing of the binding component. This is easy to analyze in the model constructed in the proof of Lemma \ref{1braidlem} by considering a pushoff given by a point $p' = (0, T')\in A$. The curve $s\mapsto (\theta = s, \tau = T', \phi = s)$ links the original braid once (positively) in the model, while the page framing on the binding is zero (corresponding to the linking between $\{\tau = 0\}$ and a pushoff $(\theta = s, \tau = \epsilon, \phi = 0)$). Hence the ``page framing'' of the braid differs by 1 from the page framing on the binding, and by our earlier choice is still lower than $\lambda$. In particular the integer $n$, representing the difference between $\lambda$ and the framing on the braid given by the lifted monodromy, is positive.

Now the proposition follows easily thanks to the monodromy description in the previous proposition.  Indeed, with the braid described as in Figure \ref{bindbraidfig}, the proposition above says that $(Y',\xi')$ is described in terms of the open book by removing a small neighborhood of $p$ in $A$, and adding $n>0$ right Dehn twists about a curve parallel to the new boundary component. This open book is precisely equivalent to the one described before Lemma \ref{legsurglemma} (with $n$ in that description replaced by $n+1$), and hence describes an (integral) inadmissible transverse surgery along the binding, or equivalently the result of positive contact surgery on the Legendrian approximation of $K$ given by a curve on $A$ parallel to the boundary.
\end{proof}

Proposition \ref{contsurgprop} completes the proof of Theorem \ref{scoopthm}, which together with Theorem \ref{capthm} implies Theorem \ref{mainthm}. 

\subsection{Strong vs weak fillability} It is natural to ask whether the filling $(Z',\omega')$ can be taken to be a strong filling, assuming $(Z,\omega)$ is strong. This is essentially a homological condition: indeed, a result of Eliashberg~\cite[Proposition~$4.1$]{eliashberg:afewremarks} shows that a weak filling can be deformed to a strong one if and only if the symplectic form is exact in a neighborhood of the boundary. With this in mind, recall that for a framed knot $K\subset Y$ and $Y'$ the result of surgery along $K$ with the given framing, the Betti numbers $b_1(Y)$ and $b_1(Y')$ differ by at most one. More specifically, a meridian $m$ of $K$, considered in $Y'$, represents an element of $H_1(Y';\arr)$ that is trivial unless both $K$ is nullhomologous and the framing is the nullhomologous framing. 
 
\begin{theorem}\label{strongfillthm} In the situation of Theorem \ref{mainthm}, suppose that $(Z, \omega)$ is a strong symplectic filling of $(Y,\xi)$, and that either of the following two conditions hold:
\begin{enumerate}
\item The symplectic form $\omega$ is exact on a neighborhood of $Y\cup \Delta$, or
\item The meridian $m$ of $K$, considered as a real homology class in $Y'$, is trivial.
\end{enumerate}
Then $(Z',\omega')$ can be deformed to a strong symplectic filling of $(Y', \xi')$.
\end{theorem}

\begin{proof}As observed above, to see that $(Z',\omega')$ can be deformed to a strong filling it suffices to show that $\omega' = \omega|_{Z'}$ is exact near $Y'$. Since $Y'$ is one boundary component of a tubular neighborhood of $Y\cup\Delta$, the first condition in the statement of the theorem clearly suffices.

For the second condition, let $N$ and $U_\Delta$ be small neighborhoods of $Y = \partial Z$ and $\Delta\subset Z$, respectively, and consider the Mayer-Vietoris sequence in de Rham cohomology
\[
H^1(V)\to H^2(N\cup U_\Delta)\to H^2(N)\oplus H^2(U_\Delta),
\]
where $V = N\cap U_\Delta$ is a neighborhood of $K\subset Y$, thickened into $Z$. Since $\omega$ is exact on $N$ and $H^2(U_\Delta) = 0$, the class $[\omega]\in H^2(N\cup U_\Delta)$ lies in the image of $H^1(V)\cong \arr$. Under the hypothesis that $[m] = 0$ in $H_1(Y')$, the composition
\[
H^1(V)\to H^2(N\cup U_\Delta)\to H^2(Y')
\]
(where the first map is the Mayer-Vietoris boundary and the second is inclusion) is the trivial map, which verifies the theorem. 

One way to check that the above composition vanishes is to pass to homology via Poincar\'e duality, and consider the composition
\[
H_3(V, \partial V) \to H_2(N\cup U_\Delta, \partial) \to H_1(Y').
\]
Since $V\cong S^1\times D^3$, the first group is generated by the relative class $[D^3, \partial D^3]$, which maps to the class $[\partial D^3]$ in the second group. Geometrically, thinking of $N$ as $Y\times I$ and $V$ as $nbd(K)\times I$, we have $D^3 = D^2\times I$ for $D^2$ a meridian disk of $K$. Thus $\partial D^3$ corresponds to $D^2 \times \{0,1\} \cup m\times I$ where $m = \partial D^2$ is the meridian. The second map above is given by intersecting chains with $Y'$, which in the case of $\partial D^3$ gives only the circle $m\times \{1\}$. Hence $[\partial D^3]\mapsto [m]$ and the claim follows.
\end{proof}

\begin{corollary}\label{strongfillcor} If $\K\subset S^3$ is a Legendrian knot in the standard contact structure on $S^3$, and $n>0$ is a given integer, then $\xi_n^-(\K)$ is weakly symplectically fillable if and only if it is strongly symplectically fillable. 
\end{corollary}

\begin{proof} If $n$ corresponds to a surgery whose smooth surgery coefficient is not zero, then the meridian of $K$ vanishes in $H_1(Y';\arr)$ (indeed, the latter group is trivial), so the result follows from case (2) of the theorem above. The case of a smooth zero-surgery arises exactly in the case that the transverse pushoff $K$ bounds a symplectic disk in a (weak) filling of $S^3$ having self-intersection zero. As seen in Section \ref{equivcondsec}, this is equivalent to $K$ bounding an embedded symplectic disk in $B^4$, whose symplectic form is exact. The conclusion follows from case (1) of Theorem \ref{strongfillthm}.
\end{proof}

\begin{table}[tb]
\begin{center}
\resizebox{\columnwidth}{!}{%
\begin{tabular}{||c|c|c|c||c|c|c|c||}
\hline
\rule{0pt}{4ex} 
\rule[-2.3ex]{0pt}{0pt}
Name  & \parbox{0.55in}{Fillable\\ surgery?} & $\mu(K)$ & Notes & Name & \parbox{0.55in}{Fillable\\ surgery?} & $\mu(K)$ & Notes\\
\hline
$3_1$ & Y & 4& Torus knot $T(3,2)$ & $10_{49}$ & Y & $\leq 12$& \\
\hline
$5_1$ & Y & 8 & & $10_{53}$ & N & $\infty$ & $c_* > g_*$ \\
\hline
$5_2$ & Y & 4 & Twist knot $K_{3}$ & $10_{55}$ & Y & $\leq 8$& \\
\hline
$7_1$ & Y & 12& & $10_{63}$ & Y & $\leq 8$&\\
\hline
$7_2$ & Y & 4 & Twist knot $K_{5}$ & $10_{66}$ & Y & $\leq 12$& \\
\hline
$7_3$ & Y & $\leq 8$ & & $10_{80}$ & Y & $\leq 12$&\\
\hline 
$7_4$ & N & $\infty$ & $c_*> g_*$ & $10_{101}$ & N & $\infty$ & $c_* > g_*$\\
\hline
$7_5$ & Y & $\leq 8$ &  & $10_{120}$ & N & $\infty$ & $c_* > g_*$\\
\hline
$8_{15}$ & Y & $\leq 8$ & & $10_{124}$ & Y& 13 & Torus knot $T(5,3)$\\
\hline
$8_{19}$ & Y & 9 & Torus knot $T(4,3)$ & $10_{126}$ & Y & 4 &\\
\hline
$8_{20}$ & Y & 0 &  & $10_{127}$ &Y & $\leq 8$ &\\
\hline
$8_{21}$ & Y & 4 & & $10_{128}$ & Y & 9 & Braid $3(12)^3(\bar{3}23)$\\
\hline
$9_1$ &Y & 16& Torus knot $T(9,2)$ & $10_{131}$ & Y & 4 & Braid $(12\bar{3}43\bar{2}\bar{1})(23\bar{2})2(12\bar{1})(\bar{4}123^2\bar{2}\bar{1}4)$ \\
\hline
$9_2$ &Y & 4 & Twist knot $K_7$ & $10_{133}$ & Y & 4& \\
\hline
$9_3$ & Y & $\leq 12$ & & $10_{134}$ & Y & $\leq 12$&\\
\hline
$9_4$ & Y & $\leq 8$ & & $10_{139}$ &Y & 13 & Braid $(12)^3(\bar{2}1^22)21$\\
\hline
$9_5$ & N & $\infty$ & $c_* > g_*$ & $10_{140}$ & Y& 0 & \\
\hline
$9_6$ & Y & $\leq 12$ & & $10_{142}$ & Y & $\leq 12$ & Braid $1^2 1^2(321\bar{2}\bar{3})1^2 2 3$\\
\hline
$9_7$ & Y & $\leq 8$ & & $10_{143}$& Y & 4& \\
\hline
$9_9$ & Y & $\leq 12$ & & $10_{145}$ & Y & $\leq 8$& Braid $32^2 1(32^2\bar{3})(21\bar{2})$\\
\hline
$9_{10}$ & N  & $\infty$ & $c_* > g_*$ & $10_{148}$ & Y & 4 & Braid $(\bar{1}2^21)(\bar{2}12)(221\bar{2}\bar{2})$\\
\hline
$9_{13}$ & N  & $\infty$ & $c_* > g_*$ & $10_{149}$ &Y & $\leq 8$ & \\
\hline
$9_{16}$ & Y & $\leq 12$ & & $10_{152}$ &Y & 13 & Braid $(12)^3 1(\bar{2}1^2 2)(21\bar{2})$\\
\hline
$9_{18}$ & Y & $\leq 8$ & & $10_{154}$ & Y & $\leq 12$ & Braid $1^2 2^2 1 (\bar{2}3 2)2^2 3$\\
\hline
$9_{23}$ & Y & $\leq 8$ & & $10_{155}$ & Y& 0 &  \\
\hline
$9_{35}$& N  & $\infty$ & $c_* > g_*$ & $10_{157}$&Y &$\leq 8$& \\
\hline
$9_{38}$& N  & $\infty$ & $c_* > g_*$ & $10_{159}$ & Y & 4& \\
\hline
$9_{45}$& Y & $\leq 8$ & & $10_{161}$& Y & $\leq 12$&\\
\hline
$9_{46}$& Y & 0 &  & $10_{165}$ & N & $\infty$ & $c_* > g_*$\\
\hline
$9_{49}$& N  & $\infty$ & $c_* > g_*$ & &  & &\\

\hline
\end{tabular}
}
\caption{Knots admitting a fillable positive contact surgery. }
\label{table:lowcross}
\end{center}
\end{table}
\section{Knots with low crossing number}\label{tablesec}

In Table \ref{table:lowcross} we tabulate those knots in the KnotInfo database with up to 10 crossings that admit fillable positive surgeries. The table includes only knots that are listed as quasipositive in the database; we do not distinguish between a knot and its mirror image (only one of these can be quasipositive, except possibly in the case of a slice knot). 

A ``Y'' in the table indicates that some transverse representative has a Legendrian approximation with a fillable positive contact surgery; such a representative is provided by either the quasipositive braid expression appearing in KnotInfo, a simple modification thereof, or by the braid indicated in the table. Where the quasipositive braid expression exhibiting the existence of a fillable surgery, as in Theorem \ref{qpslicethm}, is not obvious from the expression given on KnotInfo, we provide a braid expression using notation $1,2,3,\ldots$ and $\bar{1},\bar{2}, \bar{3},\ldots$ for the braid generators $\sigma_1,\sigma_2,\sigma_3,\ldots$ and their inverses.

For those knots with a fillable positive surgery, the table includes either a value or an upper bound for the minimal fillable surgery coefficient $\mu(K)$. Most instances of the upper bound coincide with the general one from Proposition \ref{muboundprop}, though values $\mu(K) = 0$ and $\mu(K) = 4$ are sharp in the cases $g_*(K) = 0$ and $g_*(K) = 1$, respectively, as follows from that proposition and from Proposition \ref{fillprop}. Estimates of $\mu(K)$ lower than $4g_*(K)$ arise from observing higher-multiplicity singularities in symplectic disks, as evidenced by braid (monodromy) expressions: for example, the knot $10_{128}$ admits a representation as the closure of a braid containing the expression $(12)^3$, which is the monodromy associated to an ordinary triple point singularity. When values of $\mu(K)$ are indicated as sharp, the result follows from arguments such as those elsewhere in the paper (Section \ref{surgcoeffsec} in particular). 
All examples of knots in this table that do not admit a fillable positive surgery are obstructed from doing so by Corollary \ref{c*g*cor} in that their clasp number exceeds their slice genus. Values of $c_*(K)$ in these instances were in many (perhaps all) cases obtained by Owens and Strle \cite{OwSt2016}.

\bibliography{references}
\bibliographystyle{amsplain}

\end{document}

%% file: bindbraid.pdf_tex
\begingroup%
  \makeatletter%
  \providecommand\color[2][]{%
    \errmessage{(Inkscape) Color is used for the text in Inkscape, but the package 'color.sty' is not loaded}%
    \renewcommand\color[2][]{}%
  }%
  \providecommand\transparent[1]{%
    \errmessage{(Inkscape) Transparency is used (non-zero) for the text in Inkscape, but the package 'transparent.sty' is not loaded}%
    \renewcommand\transparent[1]{}%
  }%
  \providecommand\rotatebox[2]{#2}%
  \newcommand*\fsize{\dimexpr\f@size pt\relax}%
  \newcommand*\lineheight[1]{\fontsize{\fsize}{#1\fsize}\selectfont}%
  \ifx\svgwidth\undefined%
    \setlength{\unitlength}{246.46205151bp}%
    \ifx\svgscale\undefined%
      \relax%
    \else%
      \setlength{\unitlength}{\unitlength * \real{\svgscale}}%
    \fi%
  \else%
    \setlength{\unitlength}{\svgwidth}%
  \fi%
  \global\let\svgwidth\undefined%
  \global\let\svgscale\undefined%
  \makeatother%
  \begin{picture}(1,0.62480579)%
    \lineheight{1}%
    \setlength\tabcolsep{0pt}%
    \put(0,0){\includegraphics[width=\unitlength,page=1]{bindbraid.pdf}}%
    \put(0.50034062,0.51559439){\makebox(0,0)[lt]{\lineheight{1.25}\smash{\begin{tabular}[t]{l}$(+)$\end{tabular}}}}%
    \put(0.9067086,0.5232138){\makebox(0,0)[lt]{\lineheight{1.25}\smash{\begin{tabular}[t]{l}$(-)$\end{tabular}}}}%
    \put(0,0){\includegraphics[width=\unitlength,page=2]{bindbraid.pdf}}%
    \put(0.04970163,0.00572159){\makebox(0,0)[lt]{\lineheight{1.25}\smash{\begin{tabular}[t]{l}$K$\end{tabular}}}}%
    \put(0.63639543,0.31303734){\makebox(0,0)[lt]{\lineheight{1.25}\smash{\begin{tabular}[t]{l}$p$\end{tabular}}}}%
    \put(0.37225624,0.01080116){\makebox(0,0)[lt]{\lineheight{1.25}\smash{\begin{tabular}[t]{l}$\gamma_1$\end{tabular}}}}%
    \put(0.78242022,0.00717295){\makebox(0,0)[lt]{\lineheight{1.25}\smash{\begin{tabular}[t]{l}$\gamma_2$\end{tabular}}}}%
  \end{picture}%
\endgroup%

%% file: HEmbedding.pdf_tex
\begingroup%
  \makeatletter%
  \providecommand\color[2][]{%
    \errmessage{(Inkscape) Color is used for the text in Inkscape, but the package 'color.sty' is not loaded}%
    \renewcommand\color[2][]{}%
  }%
  \providecommand\transparent[1]{%
    \errmessage{(Inkscape) Transparency is used (non-zero) for the text in Inkscape, but the package 'transparent.sty' is not loaded}%
    \renewcommand\transparent[1]{}%
  }%
  \providecommand\rotatebox[2]{#2}%
  \newcommand*\fsize{\dimexpr\f@size pt\relax}%
  \newcommand*\lineheight[1]{\fontsize{\fsize}{#1\fsize}\selectfont}%
  \ifx\svgwidth\undefined%
    \setlength{\unitlength}{223.95714227bp}%
    \ifx\svgscale\undefined%
      \relax%
    \else%
      \setlength{\unitlength}{\unitlength * \real{\svgscale}}%
    \fi%
  \else%
    \setlength{\unitlength}{\svgwidth}%
  \fi%
  \global\let\svgwidth\undefined%
  \global\let\svgscale\undefined%
  \makeatother%
  \begin{picture}(1,0.79772523)%
    \lineheight{1}%
    \setlength\tabcolsep{0pt}%
    \put(0,0){\includegraphics[width=\unitlength,page=1]{HEmbedding.pdf}}%
    \put(0.43264762,0.77023663){\makebox(0,0)[lt]{\lineheight{1.25}\smash{\begin{tabular}[t]{l}$\Delta_2$\end{tabular}}}}%
    \put(0.80102169,0.36837406){\makebox(0,0)[lt]{\lineheight{1.25}\smash{\begin{tabular}[t]{l}$\Delta_1$\end{tabular}}}}%
    \put(0.23171632,0.20093129){\makebox(0,0)[lt]{\lineheight{1.25}\smash{\begin{tabular}[t]{l}$H$\end{tabular}}}}%
    \put(0.01250839,0.46883971){\makebox(0,0)[lt]{\lineheight{1.25}\smash{\begin{tabular}[t]{l}$r_0'$\end{tabular}}}}%
    \put(0.01345913,0.66934597){\makebox(0,0)[lt]{\lineheight{1.25}\smash{\begin{tabular}[t]{l}$r_0$\end{tabular}}}}%
    \put(0.59254337,0.52673938){\makebox(0,0)[lt]{\lineheight{1.25}\smash{\begin{tabular}[t]{l}$\sigma_0$\end{tabular}}}}%
    \put(0.68500467,0.60822247){\makebox(0,0)[lt]{\lineheight{1.25}\smash{\begin{tabular}[t]{l}$\sigma_1$\end{tabular}}}}%
    \put(0,0){\includegraphics[width=\unitlength,page=2]{HEmbedding.pdf}}%
  \end{picture}%
\endgroup%